\documentclass[12pt]{amsart}

\usepackage{amsxtra,amssymb,amsthm,amsmath,amscd,url,mathrsfs}

\usepackage{fullpage}
\usepackage{color} 
\usepackage[initials]{amsrefs}
\usepackage[
        colorlinks,
]{hyperref}

\newcommand{\Cc}{\mathbb{C}}

\newcommand{\Zz}{\mathbb{Z}}
\newcommand{\Pp}{\mathbb{P}}

\newcommand{\Qq}{\mathbb{Q}}
\newcommand{\Fp}{\mathbb{F}}

\def\sums{\mathop{\sum \Bigl.^{*}}\limits}

\newcommand{\ii}{\mathrm{i}} 
\newcommand{\dd}{\mathrm{d}} 
\newcommand{\ee}{\mathrm{e}} 

\DeclareMathOperator{\GL}{GL}

\DeclareMathOperator{\Frob}{Frob}
\DeclareMathOperator{\Gal}{Gal}
\DeclareMathOperator{\Aut}{Aut}
\DeclareMathOperator{\Sym}{Sym}

\DeclareMathOperator{\rank}{rank}

\theoremstyle{plain}
\newtheorem{theorem}{Theorem}[section]
\newtheorem{corollary}[theorem]{Corollary}
\newtheorem{proposition}[theorem]{Proposition}
\newtheorem{lemma}[theorem]{Lemma}

\theoremstyle{definition}
\newtheorem{definition}[theorem]{Definition}

\newtheorem{remark}[theorem]{Remark}

\begin{document}
\title{Prime number races for elliptic curves\\ over function fields}
\author{Byungchul Cha, Daniel Fiorilli, and Florent Jouve}
\address{Department of Mathematics and Computer Science, Muhlenberg College, 2400 Chew st., Allentown, PA 18104, USA}
\email{cha@muhlenberg.edu}
\address{Department of Mathematics and Statistics, University of Ottawa, 
585 King Edward Ave, Ottawa, Ontario, K1N 6N5, Canada}
\email{daniel.fiorilli@uottawa.ca}
\address{D\'epartement de Math\'ematiques\\ B\^atiment 425\\ Facult\'e des Sciences d'Orsay\\ Universit\'e Paris-Sud 11\\
F-91405 Orsay Cedex, France}
\email{florent.jouve@math.u-psud.fr}


\begin{abstract}
We study the prime number race for elliptic curves over the function field of a proper, smooth and geometrically connected curve over a finite field. This constitutes a function field analogue of prior work by Mazur, Sarnak and the second author. In this geometric setting we can prove unconditional results whose counterparts in the number field case are conditional on a Riemann Hypothesis and a linear independence hypothesis 
on the zeros of the implied $L$-functions.
Notably we show that in certain natural families of elliptic curves, the bias generically dissipates
 as the conductor grows. This is achieved by proving a central limit theorem and combining it 
 with generic linear independence results that will appear in a separate paper.
Also we study in detail a particular family of elliptic curves that have been considered by Ulmer. 
In contrast to the generic case we show that the race exhibits very diverse outcomes, some of which are believed to be impossible in the number field setting. Such behaviors are possible in the function field case because the zeros of Hasse-Weil $L$-functions for those elliptic curves can be proven to be highly dependent among themselves, which is a very non generic situation. 
 \end{abstract}

\maketitle

\section{Introduction and statement of the main results}
\subsection{Background}

It was first noticed by Chebyshev that primes are biased in their distribution modulo $4$, in that there seems to be more primes of the form $4n+3$ than of the form $4n+1$ in initial intervals of the integers. A number of papers have been written on this phenomenon and its generalizations, and it is now known that such a bias appears in many number theoretical contexts, such as primes in arithmetic progressions, Frobenius elements in conjugacy classes of the Galois group of extensions of number fields, Fourier coefficients of modular forms, prime polynomials in residue classes over $\mathbb F_q(t)$, and so on. 

In their seminal paper \cite{RS94}, Rubinstein and Sarnak have given a framework to study questions of this type. One of the features of their work is the quantification of the so-called \emph{Chebyshev bias} in terms of an associated measure which is expressed using an explicit formula as a function of the nontrivial zeros of the involved $L$-functions.

In the case of Chebyshev's original question, Rubinstein and Sarnak determined that the logarithmic density\footnote{The logarithmic density of a set $S\subset \mathbb N$ is defined by $\delta(S):=\displaystyle\lim_{N\rightarrow \infty} \frac 1{\log N} \sum_{\substack{ n\leq N \\ n\in S}} \frac 1n$, if this limit exists.} of the set of $x$ for which $\pi(x;4,3)>\pi(x;4,1)$ exists and is given by $\delta(4;3,1) \approx 0.9959$. (Here, $\pi(x; q, a)$ is the count of primes $\le x$ that are congruent to $a$ modulo $q$.) Their results are conditional on the Generalized Riemann Hypothesis (GRH), and on the assumption (the Linear Independence hypothesis, or LI in short) that the multiset of (the ordinates of) all nontrivial zeros of the involved $L$-functions is linearly independent over $\mathbb Q$. One might think that the modulus $4$ is not exceptional here, and that there should exist other moduli $q$ and residue classes $a$ and $b$ modulo $q$ such that $\delta(q;a,b)$, the logarithmic density of the set of $x\geq 1$ for which $\pi(x;q,a)>\pi(x;q,b)$, is also very close to $1$. It turns out that as Rubinstein and Sarnak have shown, $\delta(q;a,b)$ approaches $\frac 12$ as $q\rightarrow \infty$, hence races of large moduli are very moderately biased. One can also quantify the rate of convergence here, showing for example as in \cite{FM13} that whenever $\delta(q;a,b)\neq\frac 12$, we have $|\delta(q;a,b)-1/2|=q^{-\frac 12+o(1)}$.

In the recent paper \cite{Fio12}, the second author  considered the more general race between two subsets $A$ and $B$ of the invertible reduced residues modulo $q$. It turns out that when studying the inequality $ \pi(x;q,A)> \pi(x;q,B) $ with
\[ \pi(x;q,A) := \sum_{a \in A} \pi(x;q,a), \]
things can become dramatically different from the previous case where only two residue classes were involved. Indeed, one can show under GRH and a multiplicity assumption on the zeros of $L(s,\chi)$ that there exist sequences of moduli $\{q_k\}$ and subsets $\{A_k\}$ and $\{B_k\}$ such that the associated lower and upper densities $\underline{\overline{\delta}} (q_k;A_k,B_k)$ get arbitrarily close to $1$. 
(Note also that it is known $\overline{\delta}(q;A,B) < 1$ for any $q, A, B$.) In other words, there exist `highly biased prime number races'. Under the additional assumption that LI holds, it is also proven that in order to obtain highly biased prime number races, the moduli $q_k$ need to have many prime factors, and hence highly biased prime number races are very rare.
Most races are very moderately biased, in the sense that $\delta(q;A,B)$ is usually very close to $\frac 12$.

In the context of elliptic curves, Mazur \cite{Maz08} introduced the race between the primes for which $a_p(E)$, the trace of the Frobenius at a prime $p$, is positive, against those for which $a_p(E)$ is negative. Sarnak's framework\footnote{Granville independently worked out the link between these types of prime number races and the distribution of the zeros of Hasse-Weil $L$-functions, and explained it to the second author.} in \cite{Sar07} to study this question turned out to be very effective, and explained this race very well in terms of the zeros (and potential poles) of $L(\Sym^n\!E,s)$, the symmetric power $L$-functions attached to $E$, conditional on a Riemann Hypothesis and LI. Sarnak also remarked that one can study a related race by focusing on the sign of the summatory function of $a_p(E)/\sqrt p$ using the zeros of $L(E,s)$ alone. For this race, Sarnak uncovered the influence of the analytic rank of $E$ on the bias.

Building on Sarnak's work, the second author studied in \cite{Fio13} the following question: is it possible to find highly biased prime number races in the context of elliptic curves over $\mathbb Q$, or are all races of this type only moderately biased? It turns out that conditionally on a Riemann Hypothesis and the assumption that 
the multiplicity of nonreal zeros of $L(E,s)$ is uniformly bounded (which is referred to as a bounded multiplicity assumption), the key to finding such races is to find curves $E$ whose analytic rank is significantly larger than $\sqrt{\log N_E}$, where $N_E$ is the conductor of $E$. Interestingly, the two existing conjectures (stated in~\cite{FGH07} and~\cite{Ulm02}) on the growth of the rank of elliptic curves over $\mathbb Q$ both imply the existence of the aforementioned curves. Note also that elliptic curves of large rank are extremely rare. It is widely believed that $100\%$ of the elliptic curves over $\mathbb Q$ have rank either $0$ or $1$, depending on the root number of $L(E,s)$. One can show, as is explained in \cite{Sar07}, that the bias for such curves dissipates as $N_E \rightarrow \infty$. Hence, highly biased elliptic curve prime number races over $\Qq$ are very rare. 

Coming back to the original Chebychev bias, but in the function field setting, the first author showed in~\cite{Cha08} that the framework of Rubinstein and Sarnak can be replicated in the one-variable polynomial ring $\mathbb{F}_q[t]$ over a finite field $\mathbb{F}_q$ of $q$ elements. One of the advantages in this setting is that much more is known about the (inverse) zeros of $L$-functions. First of all, the Riemann Hypothesis is known to be true in this context. Also, one can explicitly calculate the zeros of relevant $L$-functions in some specific cases and prove the function field version of LI (see the definition of Grand Simplicity Hypothesis in \cite{Cha08} and also Definition \ref{LIDefinition} below). 

The goal of the current paper is to present an unconditional 
analysis of the Chebychev bias 
for elliptic curves $E$ over a function field $K$ (in particular it can be seen as a partial $2$-dimensional generalization of~\cite{Cha08}). More precisely we study the sign of the summatory function 
of the (normalized) trace $a_v(E)/q^{\deg v/2}$ of the Frobenius at $v$, where $v$ runs over 
the places of the function field 
of a smooth proper geometrically connected curve over a finite field $\mathbb F_q$. In Section~\ref{sec:Race} 
we present the analogue of the work of Sarnak~\cite{Sar07} in our geometric setting; as in \emph{loc.~cit.}, our analysis is more 
general than what is needed to study bias phenomena in the distribution of the traces of Frobenius. 
Indeed it involves an arbitrary smooth function of the angles $\theta_v$ of the local Frobenius traces, and 
as such zeros of higher symmetric power $L$-functions come into play. 
Section~\ref{sec:Ulmer} is devoted to the study of Chebychev's bias for Ulmer's family of elliptic curves 
(defined by~\eqref{eq:UlmerCurve}). For this particular family, the Hasse--Weil $L$-function is completely explicit and can be described in an 
elementary fashion (involving e.g. multiplicative orders modulo the divisors of the parameter $d$). Finally 
in Section~\ref{sec:CLT} we prove a central limit theorem which, in conjunction with the generic linear independence results proven in~\cite{CFJ}, allows us to deduce that most elliptic curve prime number races 
over function fields are very moderately biased (see Theorem~\ref{thm:generic bias}).

\subsection*{Notations} 
Throughout the paper $p$ denotes a prime number and $\ell$ is a prime number 
different from $p$. We fix a finite field $\mathbb{F}_q$, where $q$ is a power of $p$, and a proper, smooth and geometrically connected curve
$C/\mathbb{F}_q$, with function field $K=\mathbb{F}_q(C)$. At each place $v$ of $K$ we have the residue field $k_v$ 
 which is the unique extension of $\mathbb F_q$ 
(in a fixed algebraic closure) of degree $\deg(v)$. We fix a separable closure $K_s$ of $K$ and we let 
$G_K:= \Gal(K_s/K)$ be the absolute Galois group of $K$. Finally, $E/K$ is an elliptic curve with nonconstant $j$-invariant and its analytic rank will be denoted by $\rank(E/K)$.

\subsection{Main results}\label{section:mainresults}

For a closed point $v$ of $C$ at which $E/K$ has good reduction, we let $a_v$ be the trace of Frobenius at $v$.
 We study the average behavior of $a_v/q^{\deg(v)/2} = 2 \cos \theta_v$. Note that the results of Section \ref{sec:Race} apply to the more general context of any smooth function of $\theta_v$, but we will focus on $2\cos \theta_v$ for now. We are interested in the limiting distribution arising from
\begin{equation}
 T_E(X):=-\frac X{q^{X/2}} \sum_{\substack{\deg(v) \leq X\\ v \text{ good}} } 2 \cos \theta_v = -\frac{X}{q^{X/2}}\sum_{\substack{\deg(v) \leq X\\ v \text{ good}}} \frac{a_v}{q^{\deg(v)/2}}. \label{eq:defofT(X)}
\end{equation} 
The quantity $T_E(X)$ oscillates, and usually takes both positive and negative values. To measure how long $T_E(X)$ stays positive or negative, that is, to measure its bias, we define 
\[
\underline{\overline{\delta}}(E) := \overline{\underline{\lim}}_{M\rightarrow \infty} \frac 1M \sum_{\substack{ X\leq M \\ T_E(X)> 0 }} 1.  
\]
Note that the equality between $\underline{\delta}(E)$ and $\overline{\delta}(E)$ is immediate when $T_E(X)$ is a periodic function of $X$, or when it is periodic up to an error of $o_{X\rightarrow \infty}(1)$. Further, since $E/K$ has nonconstant 
$j$-invariant, we write its Hasse-Weil $L$-function as
\[
L(E/K, T) = \prod_{j = 1}^{N_{E/K}} (1 - q \ee^{\ii \theta_j}T).
\]
Here, $N_{E/K}$ is given by the formula
\[
N_{E/K} = 4(g_C - 1) + \deg(\mathfrak{n}_{E/K}),
\]
where $g_C$ is the genus of $C$ and $\mathfrak{n}_{E/K}$ is the conductor of $E/K$. 

Our first theorem, which will be obtained by combining Corollary \ref{thm:LimitingDistribution}, Corollary \ref{thm:LimitDistForT}, 
and Theorem \ref{CentralLimitTheorem}, provides a general description of the prime number race for elliptic curves in function fields.

\begin{definition}\label{def:limdist}
We say that a function $S:[0,\pi]\rightarrow \mathbb{R}$ has a \emph{limiting distribution} if 
there exists a Borel measure $\mu$ on $\mathbb{R}$ such that for any bounded Lipcshitz continuous function $f:\mathbb{R} \longrightarrow \mathbb{R}$ we have
\begin{equation}\label{Lipschitz}
\lim_{M\to\infty}\frac1M\sum_{X=1}^M f\left(S(X)\right) = \int_{\mathbb{R}} f(t) \mathrm{d}\mu_V(t).
\end{equation}
\end{definition}

In case $S$ is chosen to be the function $T_E$ of~\eqref{eq:defofT(X)} we prove the following result.

\begin{theorem}\label{FirstTheorem} We keep the notation as above.
\begin{enumerate}
\item[(i)] The function $T_E(X)$ has a limiting distribution. Denoting by $X_E$ the associated random variable, its mean and variance are given by
\[
\mathbb{E}[X_E] =
\frac{\sqrt q}{\sqrt q-1}
\left(
\rank(E/K) - \frac12
\right),
\]
and
\[
\mathbb{V}[X_E] =
\frac14\left(\frac{\sqrt q}{\sqrt q+1}\right)^2 + 
\sums_{\theta_j\neq 0}
\frac{m(\theta_j)^2}{|1-q^{-1/2}\mathrm{e}^{-\mathrm{i}\theta_j}|^2}.
\]
Here, $m(\theta_j)$ is the multiplicity of $\theta_j$, and the starred-summation $(\sum^*)$ means that it runs over all $\theta_j\neq0$ counted without multiplicity.
\item[(ii)] Let $\{ E/K \}$ be a family of elliptic curves of unbounded conductor satisfying LI
(see Definition~\ref{LIDefinition}) and such that
$
\rank(E/K)=o(\sqrt{N_{E/K}})
$
as $N_{E/K}\to\infty$. Then, the random variable 
$$\sqrt{\frac{q-1}{q}} X_E/\sqrt{N_{E/K}} $$
converges in distribution to the standard Gaussian as $N_{E/K} \to\infty$, and as a consequence we have that $\delta(E) \rightarrow \tfrac 12$.
\end{enumerate}
\end{theorem}


This theorem says that the prime number race for elliptic curves can be generally described by its rank and the multiplicities of the zeros of $L(E/K, T)$. Indeed, up to a constant the variance is given by the square of the $2$-norm of the vector of multiplicities of the $\theta_j\neq 0$.   
Further, if we assume LI,
then the bias in the race dissipates as $N_{E/K}$ gets large, unless the rank grows faster than $\sqrt{N_{E/K}}$. These results are in line with corresponding number field counterparts in \cite{Sar07} and \cite{Fio13}. However, our results are much more unconditional than in the number field setting. This is mainly because the necessary analytic properties of $L(E/K, T)$ are established in the  function field setting. We also note that the confirmation of LI is completely conjectural in the number field case. However over function fields, it is possible to prove LI in some cases. In fact, in a separate paper \cite{CFJ}, we prove LI among certain families of elliptic curves generically by establishing quantitive bounds for the number of elliptic curves in the families satisfying LI. Let us briefly recall the construction of one of the main 
families studied in \emph{loc.~cit.}

Fix an elliptic curve $E/K$. The family we consider is a family 
of \emph{quadratic twists} of $E/K$. 
For ease of exposition let us recall the necessary definitions only in the case where $C=\Pp^1$, in which case $K$ is simply the rational function field $\Fp_q(t)$. 
Suppose that $E/K$ is given
by the Weierstrass equation $y^2=x^3+ax+b$, where $a,b\in\mathcal \Fp_q[t]$. For each $f\in K^\times$ we consider
$$
E_f\colon y^2=x^3+f^2ax+f^3b\,
$$
which is a Weierstrass equation for an elliptic curve over $K$. 
A quadratic twist 
of $E/K$ is an elliptic curve $E_f/K$ such that $f$ is not a square in $K$. 
Note that $E_f$ is isomorphic to $E_g$ over $K$ if  
and only if there exists $c\in K^\times$ such that $f=gc^2$.
 

 Let $\mathcal E\rightarrow \mathbb P^1$ be the minimal Weierstrass model (i.e.~the identity component of the N\'eron model of $E$) 
 corresponding to $E/K$.
Let us assume that $\mathcal E\rightarrow \mathbb P^1$ has at least one fiber 
of multiplicative reduction and fix a nonzero element $m\in \mathcal \Fp_q[t]$ 
which vanishes at at least one point of the locus $M$ of multiplicative reduction of 
$\mathcal E\rightarrow \mathbb P^1$.
The ``twisting family'' we consider was first introduced by Katz. It is the $(d+1)$-dimensional 
affine variety for 
which the $\Fp$-rational points are:
\begin{equation}\label{twisting-space}
\mathcal F_d(\Fp)=\{f\in \Fp[t]\colon f\text{ squarefree},\,\deg f=d,\, {\rm gcd} (f,m)=1\}\,,
\end{equation}
for any algebraic extension $\Fp\supseteq \Fp_q$, and
where $d\geq 1$ is an integer.
A remarkable fact proven by Katz is that if $f\in\mathcal F_d(\Fp_{q^n})$ then the conductor of $E_f$ only depends
 on $d$ and $q$. In particular we can let $n\rightarrow\infty$ without affecting the value of the common conductor of the twists $E_f$.


As was already mentioned, the above construction of quadratic twists can be done over any function field $K = \mathbb F_q(C)$, where $C/\Fp_q$ is any smooth geometrically connected proper curve. In particular \eqref{twisting-space} can be defined in this more general context; it will then consist of elements of $\mathcal O$, the integer ring of the compositum $\mathbb F K$.
 
The conjunction of the results of~\cite{CFJ} with Theorem \ref{CentralLimitTheorem}, which is a stronger (quantitative) version of Theorem \ref{FirstTheorem}, gives the following result which holds for any function field $K=\mathbb F_q(C)$, with $C$ as above.

\begin{theorem}
\label{thm:generic bias}
With notation as above, there exists an absolute constant $c$ such that the proportion of parameters $f\in\mathcal F_d(\mathbb F_{q^n})$ for which the inequality   
$$ \left|\delta(E_f)  - \frac 12 \right| \leq \frac{c}{\sqrt{d}} $$
fails to hold is
$
\ll_{d,E/\Fp_q(C)} n\log q/q^{nc_E d^{-2}}\,,
$
where the positive constant $c_E$ depends only on the base curve $E$.

\end{theorem}

 Since the statement of Theorem~\ref{thm:generic bias} involves 
 many different parameters, let us make several comments on 
 the way we think one should interpret it. First one fixes the piece of data $E/K$ (in particular the field of constants $\Fp_q$ of $K=\Fp_q(C)$ has fixed cardinality $q$). One should then pick $d$ large so that the first inequality means that $\delta(E_f)$ is very close to $1/2$. Then we use the remark preceding the statement of
  Theorem~\ref{thm:generic bias} to choose $n$ large so that the proportion of 
  curves excluded by the second inequality gets very close to $0$. This way 
Theorem~\ref{thm:generic bias} can be seen as a result asserting that ``generically'' prime number races for elliptic curves over function fields are very moderately biased.  

\begin{remark}\label{rem:LegendreTwists}
One may ask for the detailed description of a concrete example (i.e.~an example where one starts with a concrete base elliptic curve $E/K$) where the unspecified constants appearing in Theorem~\ref{thm:generic bias} can be made more explicit. 
Let us consider the case where $K$ is the rational function field $\Fp_q(t)$ and
 $E/K$ is the Legendre elliptic curve given by:
 $$
 y^2=x(x-1)(x-t)\,.
  $$
 The curve $E/K$ has multiplicative reduction precisely at $t$ and $t-1$ so that one 
 may take $m(t)=t(t-1)$ to define the twisting space $\mathcal F_d$.

Fix $d\geq 2$ and $\tilde{f}\in \mathcal F_{d-1}(\Fp_{q})$. 
Let $n\geq 1$ be an integer. We restrict to twists of $E$ by the $\Fp_{q^n}$-points of the open affine
 curve $U_{\tilde{f}}$ with geometric points:
$$
U_{\tilde{f}}(\overline{\Fp_q})=\{c\in\overline{\Fp_q}\colon 
(c-t)\tilde{f}(t)\in{\mathcal F_d}(\overline{\Fp_q})\}=\{c\in\overline{\Fp_q}\setminus\{0,1\}\colon 
\tilde{f}(c)\neq 0\}\,.
$$
If $c\in U_{\tilde{f}}(\Fp_{q^n})$ we denote by $E_c$ 
 the quadratic twist of $E$ by $f$ where $f(t)=(c-t)\tilde{f}(t)$. For 
 $f\in\mathcal F_d(\Fp_{q^n})$ the conductor of $E_f/K$ is 
 $2d$ (resp.~$2d-1$) if $d$ is even (resp.~if $d$ is odd). (See~\cite[Cor.~$2.2$]{CFJ} and the references therein.) Combining the arguments we develop in the proof of Theorem~\ref{thm:generic bias} (see Section~\ref{sec:CLT}) with~\cite[Cor.~$2.2$]{CFJ} we deduce that there exists an absolute constant $c_0$ such that
 $\left|\delta(E_c)  - 1/2 \right| \leq c_0d^{-1/2}$ except 
 for a proportion $\ll_{\tilde{f}} d^2n\log q/q^{n/(24d^2)}$ (where the implied constant depends only on $\tilde{f}$ and thus is independent of $n$) of exceptions in $\Fp_{q^n}$. Thus for big enough $d$ and $q$ and for the choice 
 $n=d^2$ we get densities $\delta(E_c)$ very close to $1/2$ up to a proportion of exceptions $c\in\Fp_{q^n}$ very close to $0$.
\end{remark}


One can say that Theorem~\ref{thm:generic bias} presents an orderly picture regarding the prime number races for elliptic curves in general---their bias dissipates as the conductor gets large. In contrast, our next finding shows that the races can exhibit very diverse outcomes when we look into the behaviors of $\delta(E)$ as $E$ varies in a specific family of elliptic curves. We specialize to the family Ulmer considered in \cite{Ulm02} and uncover many different and surprising prime number races. Interestingly, many of the outcomes we discover are believed to be impossible in the number field case. The reason why Ulmer's family shows such diverse results is that LI is proven to be strongly violated in this family. 

Following \cite{Ulm02}, we let $E_d$ be the elliptic curve over $\mathbb F_q[t]$ given by the Weierstrass equation
\begin{equation}\label{eq:UlmerCurve}
 y^2+xy = x^3 - t^d\,,
\end{equation}
where $d$ and $p$ are chosen so that $d\mid p^n+1$ for some $n\geq 1$. We will be interested in the associated quantity defined by \eqref{eq:defofT(X)} which we will denote by $T_d(X)$.
Our main tool for studying $T_d(X)$ will be the explicit formula given in Proposition \ref{prop:explicitformulaS_d}, from which we directly deduce that $\underline{\delta}(E_d)=\overline{\delta}(E_d)$ i.e.~the density $\delta(E_d)$ exists.

Let us first state a result asserting that an extreme bias may occur ($T_d(X)$ either taking 
mostly positive or negative values) for suitable choices of parameters.

\begin{theorem}\label{th:extremebias} For the family $\{E_d/\mathbb F_q(t)\}$, one has the following 
cases of extreme bias.
\begin{enumerate}
\item[(i)] Suppose that $q\geq 3$, and assume that either
\begin{itemize}
\item $d$ is divisible by $2$ and $q\equiv 1 \bmod 4$, or
\item $d$ is divisible by $3$.
\end{itemize}
Then, $T_d(X)>0$ for all large enough $X$, and thus $\delta(E_d)=1$.
\item[(ii)] If $q=p^k$ with $p$ large enough and $d=p^n+1$ for some $1\leq n \leq e^{q^{\frac 12}/2}$ with $n \equiv 0 \bmod k$, then $T_d(X)>0$ for all large enough $X$, and thus $\delta(E_d)=1$.
\item[(iii)] Fix $\epsilon>0$. There exists primes $d\geq 3$ and $p$ such that $p$ is a primitive root modulo $d$, and such that if we pick $q=p^{\frac{d-1}2 +1}$, then the associated curve $E_d$ has analytic rank  $1$ (resp.~$2$) if 
$(d-1)/2$ is even (resp.~odd) and
$$ 0 < \delta(E_d) < \epsilon\,.  $$
\end{enumerate}
\end{theorem}

\begin{remark}
One might wonder whether it is possible to plainly have $\delta(E_d)=0$. We will show in Corollary \ref{corollary lower bound for delta} that for $d\geq 7$ this is impossible, since we always have $\delta(E_d)> 1/2n $, where $n$ is the least positive integer such that 
$d\mid p^n+1$.
\end{remark}

The first phenomenon we uncover in Theorem \ref{th:extremebias}(i) is the existence of elliptic curves $E_d$ for which $\delta(E_d)=1$. This is quite surprising since one can show for an elliptic curve $E$ over $\mathbb Q$ that under the Riemann Hypothesis for $L(E,s)$, we always have $\delta(E)<1$.~(This follows from the analysis in  \cite[Th.~1.2]{RS94}.) In the second point of the statement (i), the integer $d$ is divisible by $3$, which creates `extra rank' for $E_d$ (see Proposition \ref{thm:UlmerLfunction}). However (ii) shows the existence of infinitely many $d$ not necessarily divisible by $3$ for which $\delta(E_d)=1$.



Part (iii) of the theorem highlights a remarkable feature of Ulmer's family. Indeed there are curves 
within the family $E_d/\mathbb F_q(t)$ of analytic rank $\geq 1$ for which $\delta(E_d)<\frac 12$. This is surprising since in the case of elliptic curves over $\mathbb Q$, Sarnak showed\footnote{When comparing our results with those of Sarnak, one should keep in mind that we are considering the race of opposite sign, that is we are considering the summatory function of $-a_v /q^{\deg v/2}$, and Sarnak is considering the summatory function of $a_p/\sqrt p$.} under a Riemann Hypothesis and a Linear Independence hypothesis that whenever the analytic rank of $E$ is greater than or equal to $1$, we have $\delta(E)> \frac 12$. By (iii) we can find prime number races which are arbitrarily biased towards negative values. Interestingly, the involved curves $E_d$ have rank at most $2$, and a high bias is quite unexpected for such curves. Indeed as was remarked by Sarnak \cite{Sar07}, one can show under GRH and LI that for elliptic curves of rank at most $2$, the density $\delta(E)$ approaches $\frac 12$ as $E$ runs over a family satisfying $N_E\rightarrow\infty$ (see~\cite[Proof of Th.~1.5]{Fio13} for a similar result with a weaker hypothesis).




Next we turn to subfamilies of $E_d/\mathbb F_q(t)$ with behavior very different to the above examples. Precisely the following statement shows the existence of curves for which there is no bias at all, in other words $\delta(E_d)=\tfrac 12$. Note that this is believed to be impossible for an elliptic curve $E/\mathbb Q$, since Sarnak \cite{Sar07} has shown under GRH and LI that $\delta(E)\neq \tfrac 12$.

 

\begin{theorem}\label{th:moderatebias}
For the family $\{E_d/\mathbb F_q(t)\}$, one has the following 
cases where $T_d(X)$ is completely unbiased. Fix $p\equiv 3 \bmod 4$ and let $d \geq 5$ be a divisor of $p^2+1$. Pick $q=p^{4k+1}$ with $k\geq 1$. Then the analytic rank of $E_d$ is either $(d-1)/4$ or $(d-2)/4$ depending on whether $d$ is congruent to $1$ or $2$ modulo $4$, and we have 
$$ \delta(E_d)=\frac 12\,. $$
\end{theorem}

Another reason why Theorem \ref{th:moderatebias} is surprising is that for elliptic curves over $\mathbb Q$, the key to producing highly biased races is to find elliptic curves $E$ for which the analytic rank is considerably larger than $\sqrt{\log N_E}$, where $N_E$ is the conductor of $E$ (see \cite[Th.~1.2]{Fio13}). However, if we pick $d=p^2+1$, then many of the curves in Theorem \ref{th:moderatebias} have very high rank, quite close to the Brumer--Mestre bound when $k$ is not too large\footnote{For elliptic curves over the rational function field $\mathbb F_q(t)$, Brumer's analogue of Mestre's bound \cite[Proposition 6.9]{Br92} states that rank$(E/K)\ll \deg (\mathfrak n_d) / \log_q \deg (\mathfrak n_d) $.}. Indeed Proposition \ref{thm:UlmerLfunction} states that the rank of $E_d$ is the left-most member of the series 
of inequalities:  
$$ \epsilon_d + \sum_{\substack{ e\mid d \\ e\nmid 6}} \frac{\phi(e)}{o_e(q)} \geq \epsilon_d+ \sum_{\substack{ e\mid d \\ e\nmid 6}} \frac{\phi(e)}{4} \geq \frac{p^2-1}{4}\,. $$
(Here we use that $q^{4} \equiv p^{4} \equiv 1 \bmod d$; the integer $\epsilon_d\in \{0,1,2,3\}$ is defined in Proposition \ref{thm:UlmerLfunction}.) This is considerably larger than $\sqrt{\text{deg}(\mathfrak n_d)}$ (which is the analog of $\sqrt{\log N_E}$), as this last quantity is given by $\sqrt{p^2+O(1)}$ (see \cite[\S10.2]{Ulm02}). In such a situation one should expect 
$\delta(E)$ to be very close to $1$, in light of \cite[Theorem 1.2]{Fio13}.  However Hypothesis BM of  \cite{Fio13}, which states that the multiplicities of the nonreal zeros of the $L$-functions associated to elliptic curves over $\mathbb Q$ are uniformly bounded, is strongly violated\footnote{This fact can actually be checked directly using Proposition \ref{thm:UlmerLfunction}.} for the elliptic curves $E_d$. This explains why no such extreme bias occurs.




  


Our final result shows that for any fixed $m\geq 1$, there are many curves for which $\delta(E_d)$ is very close to $(2m)^{-1}$. Those are races whose bias is moderate, but does not dissipate as the conductor grows.  This result is motivated by the second part of Theorem 1.1 of \cite{Fio12} where the author shows under GRH and LI that, in the context of primes in arithmetic progressions, the set of all densities $\delta(q;NR,R)$ is dense in $[\frac 12,1]$. 

\begin{theorem}
\label{thm:limitingpoints}
Define the set of all possible densities coming from Ulmer curves:
$$ S:=\{ \delta(E_d) :  d \mid p^n+1, p\geq 3, n\geq 1; q=p^k, k\geq 1 \}. $$
Then for every $m\geq 1$ there exists elements of $S$ that are arbitrarily close to $1/(2m)$, that is:
$$\{1\}\cup\{ 1/(2m)  : m\geq 1\} \subset \overline{S}. $$

\end{theorem}
The corresponding statement for elliptic curves over $\mathbb Q$ is plainly false, as the only curves having $\delta(E)< \frac 12$ are curves of rank $0$, and for these curves $\delta(E)$ approaches $\tfrac 12$ as $N_E$ tends to infinity. Moreover, it is unclear whether one should expect to have any limit points in $(\tfrac 12,1)$, given our limited knowledge on ranks. Indeed, to obtain a limit point $\eta \in (\frac 12,1)$ for the set of all $\delta(E)$ with $E$ running over the curves over $\mathbb Q$, one would need\footnote{Under the Riemann Hypothesis and LI
for the functions $L(E,s)$, this is an equivalence.} an infinite sequence of curves of analytic rank equal to $(\kappa+o(1))\sqrt{\log N_E}$, where $\kappa$ is the unique real solution to the equation 
\[ \eta = \frac 1{\sqrt{2\pi}}\int_{-\kappa}^{\infty} e^{-x^2/2}\, \mathrm{d}x. \]

\section{Limiting distributions associated with elliptic 
curves\\ over function fields}\label{sec:Race}

\subsection{Recollection on $L$-functions}\label{sec:Notations}

We keep the notation as in~\S\ref{section:mainresults}.
For a closed point $v$ of $C$, we let $a_v(E)$ be the integer defined by
\[ a_v(E) :=  q_v + 1 - \#{E}^{\mathrm{ns}}_v(k_v).\]
Here, $\#{E}^{\mathrm{ns}}_v(k_v)$ is the number of $k_v$-rational points on the nonsingular locus of the reduction $E_v$ of $E$ at $v$. Also, if $E$ has a good reduction at $v$, it is well-known that $a_v(E)$ is the trace of the $k_v$-Frobenius map on the $\ell$-adic Tate module of $E_v/k_v$. Moreover, if we let $\alpha_v$ and $\beta_v$ be its eigenvalues, then
\[
|\iota(\alpha_v)| = |\iota(\beta_v)| = {q_v}^{1/2} = q^{\deg(v)/2},
\]
for any embedding $\iota$ of $\overline{\mathbb{Q}}_{\ell}$ into the field of complex numbers (for simplicity we will omit $\iota$ from now on; its use, where needed, will be implicit). Therefore, after we fix one such embedding, there exists a unique $\theta_v$ in $[0, \pi]$ for each $v$ of good reduction such that
\[
\alpha_v = \overline{\beta_v}=q^{\deg(v)/2}\ee^{\ii\theta_v}\,.
\]

Let us define precisely what are the $L$-functions that naturally come into play in our study. We will follow the definition of~\cite[\S3.1.7]{Ulm05} to define the $L$-function $L(\rho, K, T)$ for any continuous, absolutely irreducible $\ell$-adic representation 
\[
\rho:G_K \longrightarrow \GL(V)
\]
of the absolute Galois group $G_K$ in some finite dimensional $\mathbb{Q}_{\ell}$-vector space $V$. For each $v$, we choose a decomposition group $D_v \subset G(K)$ and we let $I_v$ and $\Frob_v$ be the corresponding inertia group and the geometric Frobenius conjugacy class. Then, the $L$-function $L(\rho, K, T)$ is defined by the formal product
\begin{equation}\label{def:LFunction}
L(\rho, K, T) = \prod_v \det
\left(
1 - \rho(\Frob_v) T^{\deg v} \big| V^{\rho(I_v)}
\right)^{-1}\,,
\end{equation}
where $V^{\rho(I_v)}$ is the subspace of inertia invariants of $V$.

Of interest to us is the continuous $\ell$-adic representation
\[
\rho_{\ell,E/K} : G_K \longrightarrow \Aut(V_{\ell}(E)),
\]
arising from the Galois action on $V_\ell(E) := T_{\ell}(E) \otimes \mathbb{Q}_\ell$, where $T_\ell(E)$ is the $\ell$-adic Tate module of $E/K$. Because of a well-known independence of $\ell$ property, (namely $(\rho_{\ell,E/K})_\ell$ forms a compatible system of representations), the $L$-function $L(\rho_{\ell, E/K},K,T)$ will be denoted simply $L(E/K,T)$ in the sequel.
Its local factors are given explicitly as follows (see e.g.~\cite[Lecture $1$]{Ulm11}):
\begin{equation}\label{LocalFactor}
L(E/K, T) = 
\prod_{v \text{ good}}
(1 - a_v(E) T^{\deg(v)} + q_v T^{2\deg(v)})^{-1}
\cdot
\prod_{v \text{ bad}}
(1 - a_v(E) T^{\deg(v)})^{-1}.
\end{equation} 
Here, when $v$ is a prime of bad reduction, we have $a_v(E)=1, -1$ or $0$, depending on the reduction type of $E$ at $v$ being split multiplicative, nonsplit multiplicative or additive, respectively. Also, for each $m\ge1$, we form 
\[
\Sym^m(\rho_{\ell,E/K}): G_K \longrightarrow \Aut(\Sym^m(V_{\ell}(E))),
\]
by taking the $m$-th symmetric power of $\rho_{\ell,E/K}$. Again, by independence of $\ell$, we can and we will write $L((\Sym^mE)/K, T)$ for the $L$-function associated with $\Sym^m(\rho_{\ell, E/K})$. The local factors of $L((\Sym^mE)/K, T)$ can be described as follows. If $E/K$ has good reduction at $v$, then its local factor at $v$ is
\begin{equation}\label{eq:LocalSym}
\prod_{j=0}^m(1 - {\alpha_v}^{m-j}{\beta_v}^jT^{\deg(v)} )^{-1},
\end{equation}
whereas, for a ramified prime $v$, its local factor is
\begin{equation}\label{eq:LocalSymBad}
(1 - a_v(E)^mT^{\deg(v)} )^{-1},
\end{equation}
with, again, $a_v(E)=1, -1,$ or $0$, depending on the reduction type of $E/K$ at $v$ as before. 

Recall that $E/K$ is assumed to have nonconstant $j$-invariant. As a result, $L((\Sym^mE)/K, T)$ is a polynomial in $T$ (see~\cite[ \S3.1.7]{Ulm05}, as well as the introduction of \cite{Kat02}). More precisely, $L((\Sym^mE)/K, T) \in 1 + T\mathbb{Z}[T]$. We define $\nu_m$ to be the degree of $L((\Sym^mE)/K, T)$ and write 
\begin{equation}\label{eq:LfunctionEm}
L((\Sym^mE)/K, T) = \prod_{j=1}^{\nu_m}(1-\gamma_{m,j}T),
\end{equation}
for some complex numbers $\gamma_{m, j}$. For $m=1$, we also use the notation $N_{E/K}:= \nu_1 = \deg(L(E/K, T))$. Deligne's purity result~\cite[\S3.2.3]{Del80} implies that $\gamma_{m,j}$ is of absolute value $q^{(m+1)/2}$ under any complex embedding of $\overline{\mathbb{Q}_{\ell}}$. 
Therefore, we can define the angles $\theta_{m, j}$ by the equation
\begin{equation}\label{eq:InverseZeroForAllM}
\gamma_{m, j} = q^{(m+1)/2}\ee^{\ii \theta_{m, j}},
\end{equation}
for all $j = 1, \dots, \nu_m$ and for each $m\ge1$. Note that $\nu_m$ can be given explicitly by the formula (see~\cite[\S3.1.7]{Ulm05})
\begin{equation} \label{eq:NuM}
\nu_m = (2g_C - 2)(m+1) + \deg(\mathfrak{n}_m).
\end{equation}
Here, $g_C$ is the genus of $C/k$ and $\mathfrak{n}_m$ is the global Artin conductor of $\Sym^m(\rho_{\ell,E/K})$. We will need the following lemma, which says that $\nu_m$ grows at most linearly with $m$. 

\begin{lemma}\label{BoundArtinConductor}
There exists a positive constant $C_{E/K}$ which depends only on $E/K$ such that
\[
\nu_m \le C_{E/K}\cdot m
\]
for all $m\ge 1$.
\end{lemma}

\begin{proof}
Suppose that $E/K$ has bad reduction at $v$ and let $G: = D_v$ be the decomposition group at $v$. Thanks to the Euler characteristic formula \eqref{eq:NuM}, it is enough to show that the exponent $f_v(V)$ of the \emph{local} Artin conductor of $V:=\Sym^m(\rho_{\ell, E/K})$ at $v$ is bounded by a constant times $\deg(V) = m+1$. From Corollary $1$ of Proposition $2$ in~\cite[Chapter VI, \S2]{Ser79},
\[
f_v(V) = \sum_{i\ge 0} \frac{g_i}{g_0} \mathrm{codim}V^{G_i},
\]
where $G_i$ is the $i$-th ramification group of $G$ and $g_i=\# G_i$. Let $i_{\infty}$ be the smallest integer for which $V^{G_i}=V$  for all $i\ge i_{\infty}$, so that the above sum is a finite sum of $i_{\infty}$ terms. Then,
\[
f_v(V) \le\sum_{i=0}^{i_{\infty}} \mathrm{codim}V^{G_i} \le \sum_{i=0}^{i_{\infty}} \dim{V}  = (m+1) (i_{\infty} + 1).
\]
This completes the proof.
\end{proof}


\subsection{An explicit formula}\label{section:explicitformula}

Throughout this section we fix an elliptic curve $E$ over $K$. For readibility we do not indicate the dependency on $E$ of the objects we introduce
 whenever it is clear from context. 

Following Sarnak \cite{Sar07}, we will consider a function $V:[0, \pi] \longrightarrow \mathbb{R}$ and study the limiting distribution arising from 
\[
T_V(X):= \frac{X}{q^{X/2}}\sum_{\substack{\deg(v) \le X \\ v\text{ good}}} V(\theta_v)
\]
as $X\to\infty$. This is done by computing the Fourier expansion of $V$ using the functions 
\[
U_m(\theta) := \frac{\sin(m+1)\theta}{\sin\theta},
\]
for $m\ge0$. Indeed the family $\{ U_m\}_{m=0}^{\infty}$ forms an orthonormal basis of $L^2([0, \pi])$ with respect to the inner product
\[
\langle V_1, V_2\rangle := \frac2{\pi}\int_0^{\pi}V_1(\theta)V_2(\theta)\sin^2\theta
\, \dd\theta.
\]

Obviously computing quantities $T_{U_m}(X)$ is crucial and suffices to understand more generally $T_V(X)$ for functions $V$
that coincide with their Fourier expansion with respect to the family $\{U_m\}_{m\geq 0}$. We now proceed to computing explicitly 
$T_{U_m}(X)$ for $m\geq 1$.


The following result can be seen as a so-called ``explicit formula'' that relates a summation over primes to a summation 
over zeros of a certain $L$-function. It will be of crucial importance in the proof of the theorems stated in~\S\ref{section:mainresults}.

\begin{theorem}\label{thm:WienerIkehara}
Let $m\geq 1$ and $N\geq 1$ be integers. Let $\mathcal{E}(N) := 1$ if $N$ is even, and $\mathcal{E}(N):=0$ if $N$ is odd. Then one has:
\[
\frac{N}{q^{N/2}}\sum_{\deg(v) = N}U_m(\theta_v)
=
(-1)^{m+1}\mathcal{E}(N) - \sum_{j = 1}^{\nu_m}\ee^{\ii N \, \theta_{m, j}} + O_E(m^2 q^{-N/6})\,,
\]
\end{theorem}

To prove the theorem, we need some preliminary results.

\begin{proposition}\label{thm:Trace}
With notation as in Theorem~\ref{thm:WienerIkehara}, one has:
\begin{equation}\label{Prop22}
-
\sum_{j=1}^{\nu_m}\ee^{\ii \, N \theta_{m,j}}
=
q^{-N/2}
\sum_{d|N}
d
\sum_{\substack{\deg(v)=d \\ v\, \rm{ good}}}
U_m(\textstyle\frac Nd\theta_v) + O_E(q^{-(m+1)N/2})\,,
\end{equation}
\end{proposition}

\begin{proof}
We take the log derivative of $L((\Sym^mE)/K, T)$ using the Euler factors \eqref{eq:LocalSym} and \eqref{eq:LocalSymBad} and obtain
\begin{equation}\label{LogDerivative}
T\frac{L'}{L}((\Sym^mE)/K, T) 
= 
\sum_{d=1}^{\infty}
\sum_{\substack{\deg(v) = d \\ v \text{ good}}}
\sum_{j=0}^m
\frac{d\,{\alpha_v}^{m-j}{\beta_v}^jT^d}{1- {\alpha_v}^{m-j}{\beta_v}^jT^d}
+
\sum_{d=1}^{\infty}
\sum_{\substack{\deg(v) = d \\ v \text{ bad}}}
\frac{d\,{a_v}^mT^d}{1-{a_v}^mT^d}\,.
\end{equation}

 To simplify the summation over good primes $v$, we rewrite $U_m$ as
\begin{equation}\label{eq:UmExp}
U_m(\theta) = \sum_{j=0}^m \ee^{\ii(m-2j)\theta}\,.
\end{equation}
This gives, for any positive integer $k\ge1$, 
\begin{align*}
U_m(k\theta_v) &= 
\sum_{j=0}^m\ee^{\ii k(m-2j)\theta_v}  
= 
{q_v}^{-km/2} \sum_{j=0}^m \left({q_v}^{k(m-j)/2}\ee^{\ii k(m-j)\theta_v}\right) 
\left({q_v}^{kj/2} \ee^{-\ii kj\theta_v}\right)  \\
&=(q^{m/2})^{-dk}\sum_{j=0}^m {\alpha_v}^{k(m-j)}{\beta_v}^{kj}\,.
\end{align*}
Thus
\begin{equation}\label{eq:LocalTrace}
(q^{m/2})^{dk}U_m(k\theta_v) = \sum_{j=0}^m {\alpha_v}^{k(m-j)}{\beta_v}^{kj}\,.
\end{equation}
The summation over good primes $v$ in~\eqref{LogDerivative} becomes 
\begin{align}\label{LocalGood}
\sum_{d=1}^{\infty}
\sum_{\substack{\deg(v) = d \\ v \text{ good}}}
\sum_{j=0}^m
\frac{d\,{\alpha_v}^{m-j}{\beta_v}^jT^d}{1- {\alpha_v}^{m-j}{\beta_v}^jT^d}
&=
\sum_{d, k=1}^{\infty}
\sum_{\substack{\deg(v) = d \\ v \text{ good}}}
d
\sum_{j=0}^m
{\alpha_v}^{k(m-j)}{\beta_v}^{kj}
T^{dk} \notag \\
&=
\sum_{N=1}^{\infty}
\left(
(q^{m/2})^N
\sum_{d|N}
d
\sum_{\substack{\deg(v) = d \\ v \text{ good}}}
U_m(\textstyle\frac Nd\theta_v)
\right)
T^N,
\end{align}
where the last equality follows from \eqref{eq:LocalTrace}.
Similarly the contribution of bad primes $v$ to~\eqref{LogDerivative}  is 
\begin{align}\label{LocalBad}
\sum_{d=1}^{\infty}
\sum_{\substack{\deg(v) = d \\ v \text{ bad}}}
\frac{d\,{a_v}^mT^d}{1-{a_v}^mT^d} &=
\sum_{N=1}^{\infty}
\left(
\sum_{d|N}
d
\sum_{\substack{\deg(v) = d \\ v \text{ bad}}}
{a_v}^{mN/d}
\right)
T^N\,.
\end{align}

Since $a_v = 0$ or $\pm1$ and there are only finitely many bad primes $v$ we see that the coefficient of $T^N$ above is bounded by a constant depending on $E/K$. Using this and \eqref{LocalGood}, we simplify \eqref{LogDerivative} as follows:
\begin{equation}\label{LogDerivativeLocal}
T\frac{L'}{L}((\Sym^mE)/K, T)
= 
\sum_{N=1}^{\infty}
\left(
(q^{m/2})^N
\sum_{d|N}
d
\sum_{\substack{\deg(v) = d \\ v \text{ good}}}
U_m(\textstyle\frac Nd\theta_v)
+O_E(1)
\right)
T^N\,.
\end{equation}

On the other hand, taking the log derivative of \eqref{eq:LfunctionEm} and comparing the $N$-th coefficient of $T^N$ with \eqref{LogDerivativeLocal}, we conclude that
\[
-\sum_{j=1}^{\nu_m}{\gamma_{m, j}}^N
=
(q^{m/2})^N
\sum_{d|N}
d
\sum_{\substack{\deg(v)=d \\ v \text{ good}}}
U_m(\textstyle\frac Nd\theta_v) + O_E(1)\,.
\]

Dividing out both sides by $q^{(m+1)N/2}$ this finishes the proof of Proposition \ref{thm:Trace}.
\end{proof}

\begin{corollary}\label{thm:SolvedTrace}
With notation as in Theorem~\ref{thm:WienerIkehara}, one has:
\[
\sum_{\deg(v)=N} U_m(\theta_v) 
\ll_E m\frac{q^{N/2}}{N}\,.
\]
\end{corollary}

\begin{proof}
In \eqref{Prop22}, we split out the term $d=N$ to get
\begin{equation}\label{eq:SolvedTrace}
\frac{N}{q^{N/2}}\sum_{\deg(v)=N} U_m(\theta_v) 
=
- \sum_{j=1}^{\nu_m}\ee^{\ii \, N \theta_{m,j}}
-
q^{-N/2}
\sum_{\substack{d|N \\ d\le N/2}}
d
\sum_{\substack{\deg(v)=d \\ v \text{ good}}}
U_m(\textstyle\frac Nd\theta_v) + O_E(q^{-N/2})\,.
\end{equation}

Then, it is enough to show that the right side is bounded by $(m+1)$ times an absolute constant. From Lemma \ref{BoundArtinConductor}, we have
\begin{equation}\label{eq:1}
\left|
-\sum_{j=1}^{\nu_m}\ee^{\ii \, N \theta_{m,j}} 
\right|
\le C_{E/K}m\,.
\end{equation}

To bound the rest of the summation in the right side of \eqref{eq:SolvedTrace}, we use the trivial bound 
\begin{equation}\label{UmTrivialBound}
|U_m(\theta)| \le m+1,
\end{equation}
which is immediate from \eqref{eq:UmExp}, and the count of places of fixed degree in $K=\mathbb{F}_q(C)$ (see e.g~\cite[Prop.~$6.3$]{Br92})
which yields the estimate
\begin{equation}\label{PrimeNumberTheorem}
\Big|\big(\sum_{\deg(v) = d}1\big) -\frac{q^d}d\Big|\leq \frac{2g_C+1}{1-q^{-1}}q^{d/2}\,.
\end{equation}

The second term on the right hand side of \eqref{eq:SolvedTrace} is then bounded above in absolute value by
$$ (m+1) q^{-N/2}
\sum_{\substack{d|N \\ d\le N/2}}
d
\sum_{\substack{\deg(v)=d \\ v \text{ good}}} 1 \ll_C m q^{-N/2}\sum_{\substack{d|N \\ d\le N/2}}q^d \ll m.  $$

The proof follows.
%
%
\end{proof}

 We are now ready to prove Theorem~\ref{thm:WienerIkehara}.

\begin{proof}[Proof of Theorem \ref{thm:WienerIkehara}]
We begin with \eqref{eq:SolvedTrace} and separate out the term $d=N/2$, which exists only when $N$ is even. (Recall, by definition, $\mathcal{E}(N) = 1$ if $N$ is even and $0$ otherwise.)
\begin{align}\label{eq:UmTotal}
\frac{N}{q^{N/2}} \sum_{\deg(v)=N} U_m(\theta_v) 
=&
- \sum_{j=1}^{\nu_m}\ee^{\ii \, N \theta_{m,j}} 
-
\frac1{q^{N/2}}
\sum_{\substack{d|N \\ d\le N/2}}
d
\sum_{\deg(v)=d}
U_m(\textstyle\frac Nd\theta_v) + O_E(q^{-N/2}) \\
=&
- \sum_{j=1}^{\nu_m}\ee^{\ii \, N \theta_{m,j}} 
-
\mathcal{E}(N)
\frac{N/2}{q^{N/2}}
\sum_{\deg(v)=N/2}
U_m(2\theta_v) \nonumber\\
  &-
\frac1{q^{N/2}}
\sum_{\substack{d|N \\ d\le N/3}}
d
\sum_{\substack{\deg(v)=d \\ v \text{ good}}}
U_m(\textstyle\frac Nd\theta_v) + O_E(q^{-N/2}) \,.\nonumber
\end{align}

First, we handle the terms with $d\le N/3$. From \eqref{UmTrivialBound} and \eqref{PrimeNumberTheorem} again, 
\begin{align}\label{eq:DLessThanNOver3}
\left|\frac1{q^{N/2}}
\sum_{\substack{d|N \\ d\le N/3}}
d
\sum_{\substack{\deg(v)=d \\ v \text{ good}}}
U_m(\textstyle\frac Nd\theta_v) 
\right|
&\le
\displaystyle
\frac{m+1}{q^{N/2}}
\sum_{\substack{d|N \\ d\le N/3}}
d
\left(
\frac{q^d}{d} + O_C\left(\frac{q^{d/2}}{d}\right)
\right) \\ 
&\ll 
m q^{-N/6}\,.\nonumber 
\end{align}
Next, to estimate $\sum_{\deg(v)=N/2}U_m(2\theta_v)$, we consider the following Fourier expansion of $V(\theta):= U_m(2\theta)$ with respect to the orthonormal basis $\{U_k\}_{k=0}^{\infty}$
\[
V(\theta) 
= 
\sum_{k=0}^\infty \langle V, U_k \rangle U_k(\theta).
\]

One has:
\begin{align*}
\sum_{\deg(v)=N/2} 
V(\theta_v) 
&=
\sum_{k=0}^\infty \langle V, U_k \rangle
\sum_{\deg(v)=N/2} 
U_k(\theta_v) \\
&= \langle V, U_0 \rangle 
\sum_{\deg(v)=N/2} 
1
+
\sum_{k=1}^\infty \langle V, U_k \rangle
\sum_{\deg(v)=N/2} 
U_k(\theta_v)  \\
&=
\langle V, U_0 \rangle
\left( \frac{q^{N/2}}{N/2} + O_C(q^{N/4}/N) \right)
+
\sum_{k=1}^\infty \langle V, U_k \rangle
\sum_{\deg(v)=N/2}
U_k(\theta_v)  \,.
\end{align*}
Note that Lemma \ref{thm:FourierExpansion} below provides the Fourier coefficients of $V$ explicitly. In particular,
\begin{equation}\label{eq:U0}
\langle V, U_0 \rangle
= (-1)^m.
\end{equation}

Moreover, Corollary \ref{thm:SolvedTrace} gives
\[
\sum_{\deg(v)=N/2}
U_k(\theta_v)  
\ll_E 
k\frac{q^{N/4}}{N}\,.
\]
 Again, by Lemma \ref{thm:FourierExpansion}, we can determine $\langle V, U_k \rangle$ for all $k\ge1$. We deduce
\begin{align}
\sum_{k=1}^\infty \langle V, U_k \rangle
\sum_{\deg(v)=N/2}
U_k(\theta_v)  
&\ll_E
\frac{q^{N/4}}N\sum_{k=1}^\infty |\langle V, U_k \rangle|
k
\notag \\
&\ll_E
\frac{q^{N/4}}N m^2.
\label{eq:HigherV}
\end{align}
From \eqref{eq:U0} and \eqref{eq:HigherV}, we obtain
\begin{equation}\label{eq:DNover2}
\frac{N/2}{q^{N/2}} \sum_{\deg(v)=N/2} U_m(2\theta_v)  
=(-1)^m + O_E(m^2q^{-N/4})\,.
\end{equation}

Collect \eqref{eq:DLessThanNOver3} and \eqref{eq:DNover2} and put them into \eqref{eq:UmTotal}, to obtain
\[
\frac{N}{q^{N/2}} \sum_{\deg(v)=N} U_m(\theta_v) 
=
- \sum_{j=1}^{\nu_m}\ee^{\ii \, N \theta_{m,j}} 
-
\mathcal{E}(N)(-1)^m + O_E(m^2q^{-N/6})\,.
\]
This finishes the proof of Theorem \ref{thm:WienerIkehara}.
\end{proof}

\begin{lemma}\label{thm:FourierExpansion} 
For $m\ge0$, 
\[
U_m(2\theta) = U_{2m}(\theta) - U_{2m-2}(\theta) + \cdots+(-1)^{m+1}U_2(\theta) + (-1)^m.
\]
\end{lemma}
\begin{proof}
Using \eqref{eq:UmExp}, one can prove that
\[ U_{m+1}(2\theta) + U_m(2\theta) = U_{2(m+1)}(\theta).  \]
Now, an easy induction on $m$ completes the proof. 
\end{proof}

 \subsection{Limiting distribution arising from smooth functions}

We now derive from Theorem~\ref{thm:WienerIkehara} a decomposition of $T_V(X)$ (under suitable hypotheses on the 
function $V$) from which we deduce the existence of a limiting distribution for $T_V(X)$.


Let $V:[0, \pi]\longrightarrow \mathbb{R}$ be a function and let $V_m:= \langle V, U_m \rangle$ be the $m$-th Fourier coefficient of $V$. 
To ensure convergence, we will assume that $V_m \ll m^{-3-\eta}$ for some $\eta > 0$. 
In particular this will guarantee that the Fourier expansion
\[
V(\theta) = \sum_{m=1}^{\infty} V_m U_m(\theta)
\]
converges (uniformly and absolutely) for $\theta \in [0, \pi]$, by the trivial bound $|U_m(\theta)|\leq m+1$. The reason why we require such a strong decay rate for $V_m$ is that 
 we also need to ensure the convergence of the error term in~\eqref{eq:ComputeTV}.

Assuming further that $\langle V, U_0 \rangle = 0$ we may apply Theorem \ref{thm:WienerIkehara} to get:
\begin{align}
T_V(X) &=
\frac{X}{q^{X/2}}
\sum_{N = 1}^X 
\sum_{\substack{\deg v = N\\ v \text{ good}}} 
V(\theta_v)
= 
\frac{X}{q^{X/2}}
\sum_{N = 1}^X \sum_{m=1}^{\infty} V_m 
\sum_{\substack{\deg v = N\\ v \text{ good}}} 
U_m(\theta_v)\nonumber \\
&= 
\frac{X}{q^{X/2}}
\sum_{N = 1}^X 
\frac{q^{N/2}}{N}
\left(
\mathcal{E}(N)\sum_{m=1}^{\infty} V_m (-1)^{m+1}
-
\sum_{m=1}^{\infty}V_m
\sum_{j = 1}^{\nu_m} \ee^{\ii N\theta_{m, j}} + 
O\left(\sum_{m=1}^{\infty}m^2 V_mq^{-N/6}\right)
\right)\,.\label{eq:ComputeTV}
\hspace{-1cm}
\end{align}
We break the last line into:
\[
T_V(X) = 
T^{(\mathrm{I})}_V(X)
+
T^{(\mathrm{II})}_V(X)
+O(X(\log X)^{-1}q^{-X/6}),
\]
where 
\begin{align*}
T^{(\mathrm{I})}_V(X) &:=
\frac{X}{q^{X/2}}
\sum_{N = 1}^X 
\frac{q^{N/2}}{N}
\mathcal{E}(N)
\left(
\sum_{m=1}^{\infty}
V_m(-1)^{m+1}
\right),\\
T^{(\mathrm{II})}_V(X)
&:=
-\frac{X}{q^{X/2}}
\sum_{N = 1}^X 
\frac{q^{N/2}}{N}
\left(
\sum_{m=1}^{\infty}
V_m\sum_{j = 1}^{\nu_m} \ee^{\ii N \theta_{m, j}}
\right).
\end{align*}
To simplify  $T^{(\mathrm{I})}_V(X)$ and $T^{(\mathrm{II})}_V(X)$
 further, we use~\cite[Cor.~$2.3$ and Cor.~$2.4$]{Cha08}. Since the statements are quite short we recall them without proof in the following lemma.
   
\begin{lemma}\label{EstimateOfDN}
Let $N\geq 1$ be an integer and let $\mathcal E(N)$ be as above. The following holds.
\begin{enumerate}
\item[(i)] Let
\begin{equation}\label{eq:cplusminus}
c_{\pm}(X) := 
\begin{cases}
q/(q-1) & \text{ for even $X$}, \\
\sqrt{q}/(q-1) & \text{ for odd $X$}.
\end{cases}
\end{equation}
Then
\[
 \frac{X}{q^{X/2}}\sum_{N=1}^X \mathcal{E}(N)\frac{q^{N/2}}{N} =c_{\pm}(X)+o(1)\,,
\]
as $X\rightarrow \infty$.
\item[(ii)] Let $\gamma:=\sqrt{q}\ee^{\ii \theta}$ be a complex number of argument $\theta\in[0,2\pi]$.
 Then\footnote{In~\cite{Cha08} this is stated under the assumption $\theta\in[0,\pi]$ but the conclusion obviously holds 
 more generally for any $\theta\in[0,2\pi]$ simply by applying complex conjugation.}
\[
\frac{X}{q^{X/2}}\sum_{N=1}^X\frac{\gamma^N}{N}=\frac{\gamma\ee^{\ii \theta X}}{\gamma-1}+o(1)\,,
\]
as $X\rightarrow\infty$.
\end{enumerate}
\end{lemma}

Applying (i) of the lemma we obtain
\begin{equation}\label{eq:SI}
T^{(\mathrm{I})}_V(X)
=
c_{\pm}(X) \sum_{m=1}^{\infty}(-1)^{m+1}V_m + o(1)\,.
\end{equation}

For $T^{(\mathrm{II})}_V(X)$, we apply (ii) of the lemma with $\gamma = \gamma_{m, j}/q^{m/2}$ (see \eqref{eq:InverseZeroForAllM}). Then, as $X\to\infty$,
\[
\frac{X}{q^{X/2}}
\sum_{N=1}^X
\frac{q^{N/2}}{N}
\ee^{\ii N \theta_{m, j}}
=
\frac{X}{q^{X/2}}
\sum_{N=1}^X
\frac{(\gamma_{m, j}/q^{m/2})^N}{N}
=
\frac{\gamma_{m, j}}{\gamma_{m, j}-q^{m/2}}\ee^{\ii \theta_{m, j}X} + o(1).
\]
Thus, 
\begin{equation}\label{eq:SII}
T^{(\mathrm{II})}_V(X)
=
\sum_{m=1}^{\infty}
V_m
\left(
\sum_{j = 1 }^{\nu_m} 
\frac{\gamma_{m, j}}{\gamma_{m, j}-q^{m/2}}\ee^{\ii \theta_{m, j}X}
\right)
+ o(1)\,,
\end{equation}
which converges absolutely by our assumption that $V_m \ll m^{-3-\eta}$.

Next we separate out the terms with $\gamma_{m, j} = q^{(m+1)/2}$, or equivalently 
$\theta_{m, j} = 0$. To do so, define
\begin{equation}\label{MmDefinition}
M_m(1) := \#\{ j \mid \theta_{m, j}=0 \text{ with } j=1, \dots, \nu_m\}.
\end{equation}
In other words, $M_m(1)$ is the multiplicity of the zero $T=q^{-(m+1)/2}$ in $L((\Sym^mE)/K, T)$. Then, from \eqref{eq:SII},
\begin{equation}\label{SIIAgain}
T^{(\mathrm{II})}_V(X)
=
-\sum_{m=1}^{\infty}
V_m\frac{\sqrt q}{\sqrt q - 1}M_m(1)
-\sum_{m=1}^{\infty}
V_m
\sum_{\substack{j = 1, \dots, \nu_m \\ \theta_{m,j}\neq 0}} 
\frac{\gamma_{m, j}}{\gamma_{m, j}-q^{m/2}}\ee^{\ii \theta_{m, j}X}
+ o(1).
\end{equation}
Combining \eqref{eq:SI} and \eqref{SIIAgain} we obtain the following result.

\begin{proposition}\label{thm:GeneralSV}
For a function $V:[0, \pi] \longrightarrow \mathbb{R}$ with $\langle V, U_0 \rangle = 0$ and $V_m \ll m^{-3-\eta}$ for some $\eta>0$, we have that
\[
T_V(X) = Q_V(X) + R_V(X) +o_{X\rightarrow \infty}(1)
\]
where
\begin{align*}
Q_V(X):= &
\sum_{m=1}^{\infty}
\left(
(-1)^{m+1}c_{\pm}(X) - \frac{\sqrt q}{\sqrt q - 1}M_m(1)
\right)
V_m\,,\\
R_V(X):=&
-\sum_{m=1}^{\infty}
V_m
\sum_{\substack{j = 1, \dots, \nu_m \\ \theta_{m,j}\neq 0}} 
\frac{\gamma_{m, j}}{\gamma_{m, j}-q^{m/2}}\ee^{\ii \theta_{m, j}X}\,.
\end{align*}
Here $c_{\pm}(X)$ and $M_m(1)$ are defined as in \eqref{eq:cplusminus} and \eqref{MmDefinition}. 
\end{proposition}
\begin{corollary}\label{thm:LimitingDistribution}
Let $V$ be as in Proposition~\ref{thm:GeneralSV}. The quantity $T_V(X)$ has a limiting distribution  $\mu_V$ in the sense of Definition~\ref{def:limdist}. Moreover, for $k\ge1$,
\begin{equation}\label{Moment}
\lim_{M\to\infty}\frac1M\sum_{X=1}^M T_V(X)^k = \int_{\mathbb{R}} t^k\mathrm{d}\mu_V(t).
\end{equation}
\end{corollary}
\begin{proof}
This is the function field analogue of~\cite[Lemmas 2.3 and 2.5]{Fio13}. We just need to make a slight adaption to handle the difference arising from the fact that our definition of limiting distribution (Definition~\ref{def:limdist}) uses a (discrete) summation, rather than an integral. 

First, we outline the proof of existence of a limiting distribution. This is well-known in the number field setting, originally described in \cite{RS94} and in \cite{ANS14} for more general situations. Following \cite{ANS14}, we say that a real-valued function $\phi(X)$ defined for all positive integers $X$ is called a \emph{$B^2$-almost periodic} function if for any $\epsilon>0$, there exists a real-valued trigonometric polynomial
\begin{equation}\label{TrigPoly}
P_{N(\epsilon)}(X) 
= \sum_{n=1}^{N(\epsilon)} r_n(\epsilon) \ee^{\ii \lambda_n(\epsilon) \, X}
\end{equation}
such that 
\begin{equation}\label{Bsquare}
\limsup_{Y\to\infty} \frac1{Y}\sum_{X=1}^Y |\phi(X) - P_{N(\epsilon)}(X)|^2 <\epsilon^2.
\end{equation}
(Here, $\lambda_n(\epsilon)$ is real and $r_n(\epsilon)$ is complex.) Using a discrete version of the Kronecker-Weyl theorem, it can be shown that any trigonometric polynomial of the form \eqref{TrigPoly} has a limiting distribution (in the sense of Definition~\ref{def:limdist}.) Also, an obvious adaptation of~\cite[Th.~2.9]{ANS14} shows that any $B^2$-almost periodic function has a limiting distribution. Next, for any (large) positive number $M$, we define (using the notation of 
Proposition~\ref{thm:GeneralSV})
\[
\mathcal{E}_V(X, M) = -\sum_{m \ge M}
V_m
\sum_{\substack{j = 1, \dots, \nu_m \\ \theta_{m,j}\neq 0}} 
\frac{\gamma_{m, j}}{\gamma_{m, j}-q^{m/2}}\ee^{\ii \theta_{m, j}X} + o_{X\rightarrow \infty}(1),
\]
so that $T_V(X)$ is a sum of a trigonometric polynomial and $\mathcal{E}_V(X, M)$. Using a trivial bound we easily obtain
\begin{equation}\label{ErrorBound}
|\mathcal{E}_V(X, M)| \le 
\frac{\sqrt q}{\sqrt q - 1} \sum_{m\ge M}
|V_m|\nu_m + o_{X\rightarrow \infty}(1).
\end{equation}
The decay condition on $V_m$ now implies its mean square can be made arbitrarily small by choosing $M$ large and we deduce that $T_V(X)$ is $B^2$-almost periodic, thus has a limiting distribution, say, $\mu_V$. In fact, a straightforward calculation shows that $T_V(X)$ is bounded for all $X$, therefore, $\mu_V$ is supported on a bounded set in $\mathbb{R}$. To prove \eqref{Moment}, as 
in~\cite[Lemma $2.5$]{Fio13}, we choose a Lipschitz continuous function $f$ equal to $t^k$ on a set containing the support of $\mu_V$ and to zero outside and apply \eqref{Lipschitz}. We omit the details.
\end{proof}

\subsection{Limiting distribution arising from $T_E(X)$}
In this subsection, we look into the special case $T_E(X)= T_V(X)$ i.e.~$V(\theta)=-U_1(\theta) = -2\cos\theta$, as 
defined in \eqref{eq:defofT(X)}. Thus we have $V_m = 0$ for all $m\ge 2$. Recall from \S\ref{sec:Notations} that
 $L(E/K, T)$ is a polynomial in $T$ of degree $\nu_1= N_{E/K}$. We will write its inverse zeros as 
$\gamma_j = q \ee^{\ii \theta_j}$ for $j = 1, \dots, N_{E/K}$, so that
\begin{equation}\label{CaseMEqual1}
L(E/K, T) = \prod_{j = 1}^{N_{E/K}} (1 - q \ee^{\ii \theta_j}T).
\end{equation}
We note from the definition \eqref{MmDefinition} that $M_1(1)= \mathrm{rank}(E/K)$, the analytic rank of $E/K$. 
Proposition \ref{thm:GeneralSV} gives the following statement in the case $V=-U_1$.

\begin{corollary}\label{thm:LimitDistForT}
With notation as above we have
\begin{equation}\label{LimitDistForT}
T_E(X) = Q_E(X) + R_E(X) + o_{X\to\infty}(1),
\end{equation}
where
\begin{align*}
Q_E(X):=& 
\frac{\sqrt q}{\sqrt q - 1}
\mathrm{rank}(E/K) - c_{\pm}(X)
\,,\\
R_E(X):=&
\sum_{\substack{j = 1, \dots, N_{E/K} \\ \theta_j\neq 0}} \frac{1}{1 - q^{-1/2}\ee^{-\ii\theta_j}}\ee^{\ii \theta_jX}=
\sum_{\substack{j = 1, \dots, N_{E/K} \\ \theta_j\neq 0}} \frac{\gamma_j}{\gamma_j - \sqrt q}\ee^{\ii \theta_jX}\,.
\end{align*}
\end{corollary}
\begin{proof}
This is obvious from Proposition \ref{thm:GeneralSV} because $V_1 = -1$ and $V_m = 0$ for all $m\ge 2$. 
\end{proof}

%

Corollary \ref{thm:LimitingDistribution} applied to the case $V= -U_1$ enables us to study the random variable
$X_E$ associated to the limiting distribution of $T_E(X)$. Theorem~\ref{FirstTheorem}(i) shows that we can obtain simple closed formul\ae\, for the mean and variance of $X_E$. 


\begin{proof}[Proof of Theorem~\ref{FirstTheorem}(i)]
First, we compute \[
\mathbb{E}[X_E] =
\lim_{M\to\infty}
\frac1M
\sum_{X=1}^M T_E(X).
\]
From Corollary \ref{thm:LimitDistForT}, we have
\begin{equation}\label{eq:SE1}
T_E(X) = 
\frac{\sqrt q}{\sqrt q -1}\mathrm{rank}(E/K) 
-c_{\pm}(X)
+
\sum_{\theta_j\neq0} \frac{1}{1-q^{-1/2} \mathrm{e}^{-\mathrm{i}\theta_j}} \mathrm{e}^{\mathrm{i} \theta_j X}+o_{X\rightarrow \infty}(1).
\end{equation}
It is easy to show that
\[
\lim_{M\to\infty}
\frac1M
\sum_{X=1}^M 
\left(
\frac{\sqrt q}{\sqrt q -1}\mathrm{rank}(E/K)
-c_{\pm}(X)
\right)
=
\frac{\sqrt q}{\sqrt q-1}
\left(
\rank(E/K) - \frac12
\right).
\]
In addition, for any $\theta\in (0,2\pi)$,
\begin{equation}\label{eq:EquiDistributionCircle}
\sum_{X=1}^M \mathrm{e}^{\mathrm{i}\theta X} = O\left( \frac 1{\| \theta/2\pi \| } \right),
\end{equation}
where $\lVert \cdot\rVert$ denotes the distance to the nearest integer. The formula for $\mathbb{E}[X_E]$ follows immediately from \eqref{eq:SE1}. 
For $\mathbb{V}[X_E]$, we must compute
\[
\mathbb{V}[X_E] = \lim_{M\to\infty}\frac1M
\sum_{X=1}^M
(T_E(X) - \mathbb{E}[X_E])^2.
\]
Note that
$
{\sqrt q}/{(2(\sqrt q - 1))} - c_{\pm}(X) =
\pm{\sqrt q}/{(2(\sqrt q + 1))}
$,
with a $+$ (resp. $-$) sign if $X$ is odd (resp. even). Thus
\begin{multline}\label{eq:SE2}
\left(
\frac12\frac{\sqrt q}{\sqrt q - 1} - c_{\pm}(X) +
\sum_{\theta_j\neq0} \frac{1}{1-q^{-1/2} \mathrm{e}^{-\mathrm{i}\theta_j}} \mathrm{e}^{\mathrm{i} \theta_j X}
\right)^2
= \\
\frac14
\left(
\frac{\sqrt q}{\sqrt q+1}
\right)^2
\pm\frac{\sqrt q}{\sqrt q + 1}
\sum_{\theta_j\neq0} \frac{1}{1-q^{-1/2} \mathrm{e}^{-\mathrm{i}\theta_j}} \mathrm{e}^{\mathrm{i} \theta_j X}
+
\left(
\sum_{\theta_j\neq0} \frac{1}{1-q^{-1/2} \mathrm{e}^{-\mathrm{i}\theta_j}} \mathrm{e}^{\mathrm{i} \theta_j X}
\right)^2.
\end{multline}
Using this and \eqref{eq:SE1}, we obtain
\begin{align*}
\mathbb{V}[X_E] &= 
\lim_{M\to\infty}\frac1M
\sum_{X=1}^M
\left(
\frac12\frac{\sqrt q}{\sqrt q - 1} - c_{\pm}(X) +
\sum_{\theta_j\neq0} \frac{1}{1-q^{-1/2} \mathrm{e}^{-\mathrm{i}\theta_j}} \mathrm{e}^{\mathrm{i} \theta_j X}
\right)^2
 \\
&=
\frac14
\left(
\frac{\sqrt q}{\sqrt q+1}
\right)^2
+
\lim_{M\to\infty}\frac1M
\sum_{X=1}^M
\left(
\sum_{\theta_j\neq0} \frac{1}{1-q^{-1/2} \mathrm{e}^{-\mathrm{i}\theta_j}} \mathrm{e}^{\mathrm{i} \theta_j X}
\right)^2,
\end{align*}
where the last line follows from \eqref{eq:SE2} and \eqref{eq:EquiDistributionCircle}. Next we compute
\begin{align*}
\sum_{X=1}^M
\left(
\sum_{\theta_j\neq0} \frac{1}{1-q^{-1/2} \mathrm{e}^{-\mathrm{i}\theta_j}} \mathrm{e}^{\mathrm{i} \theta_j X}
\right)^2
&=
\sum_{X=1}^M
\left(
\sums_{\theta_j\neq0} \frac{m(\theta_j)\mathrm{e}^{\mathrm{i} \theta_j X}}{1-q^{-1/2} \mathrm{e}^{-\mathrm{i}\theta_j}} 
\right)^2 \\
&=
\sum_{X=1}^M
\left(
\sums_{\theta_{k}\neq0} \frac{m(\theta_{k})\mathrm{e}^{\mathrm{i} \theta_{k} X}}{1-q^{-1/2} \mathrm{e}^{-\mathrm{i}\theta_{k}}} 
\right) 
\left(
\sums_{\theta_{l}\neq0} \frac{m(\theta_{l})\mathrm{e}^{-\mathrm{i} \theta_{l} X}}{1-q^{-1/2} \mathrm{e}^{\mathrm{i}\theta_{l}}} 
\right)\,. 
\end{align*}
Splitting out the diagonal term, the right hand side equals
$$
M
\sums_{\theta_j\neq 0}\frac{m(\theta_j)^2}{|1-q^{-1/2}\mathrm{e}^{-\mathrm{i}\theta_j}|^2}
+
\sums_{\substack{\theta_k, \theta_l\neq 0 \\ k\neq l}}
\frac{m(\theta_k)m(\theta_l)}{(1-q^{-1/2}\mathrm{e}^{-\mathrm{i}\theta_k})(1-q^{-1/2}\mathrm{e}^{\mathrm{i}\theta_l})} \sum_{X=1}^M\mathrm{e}^{\mathrm{i}(\theta_k-\theta_l)X}\,.
$$
We divide the last line by $M$, let $M\to\infty$, and use \eqref{eq:EquiDistributionCircle} one more time to obtain
\[
\mathbb{V}[X_E] = 
\frac14
\left( \frac{\sqrt q}{\sqrt q+1} \right)^2
+
\sums_{\theta_j\neq 0}\frac{m(\theta_j)^2}{|1-q^{-1/2}\mathrm{e}^{-\mathrm{i}\theta_j}|^2}.
\]
This concludes the proof.
\end{proof}



%

\section{Ulmer's family}\label{sec:Ulmer}
The goal of this section is to prove the results stated in~\S\ref{section:mainresults} regarding the function $T_E(X)$ associated to the elliptic curves of Ulmer's family in \cite{Ulm02}. 
Let $\mathbb{F}_q(t)$ be the rational function field over $\mathbb{F}_q$. Following \cite{Ulm02}, we define $E_d/\mathbb{F}_q(t)$ to be the elliptic curve over $\mathbb F_q(t)$ given by the Weierstrass equation
\[ E_d: y^2+xy = x^3 - t^d,\]
and we write $T_d(X)$ for the function $T_E(X)$ that arises from $E_d/\mathbb{F}_q(t)$. Essential to us is the following explicit description of their Hasse-Weil $L$-function.
\begin{proposition}\label{thm:UlmerLfunction}
Suppose that $d$ divides $p^n + 1$ for some $n$, and let $L(E_d/\mathbb{F}_q(t),T)$ be the Hasse-Weil $L$-function of $E_d$ over $\mathbb{F}_q(t)$. Then,
\[
L(E_d/\mathbb{F}_q(t), T) =
(1-qT)^{\epsilon_d} \prod_{\substack{e|d \\ e\nmid 6}}
\left(1-(qT)^{o_e(q)}\right)^{\phi(e)/o_e(q)}.
\]
Here, $\phi(e)=\#(\mathbb{Z}/e\mathbb{Z})^*$ is the Euler-phi function and $o_e(q)$ is the (multiplicative) order of $q$ in $(\mathbb{Z}/e\mathbb{Z})^*$. Further, $\epsilon_d$ is defined as
\[
\epsilon_d :=
\begin{cases}
0 & \text{ if $2\nmid d$ or $4 \nmid q-1$} \\
1 & \text{ if $2|d$ and $4 | q-1$}
\end{cases}
+
\begin{cases}
0 & \text{ if $3\nmid d$}  \\ 
1 & \text{ if $3|d$ and $3\nmid q-1$}  \\ 
2 & \text{ if $3|d$ and $3|q-1$}
\end{cases}\,.
\]
In particular we have an explicit formula for the analytic rank of $E_d/\mathbb{F}_q(t)$:
$$
{\rm rank}(E/\mathbb{F}_q(t))= \epsilon_d + \sum_{\substack{ e\mid d \\ e\nmid 6}} \frac{\phi(e)}{o_e(q)}\,.
$$
\end{proposition}

\begin{proof}
This is essentially a corollary of the main results from \cite{Ulm02}. To briefly highlight the main ingredients,~\cite[Cor.~7.7 and Prop.~8.1]{Ulm02} computes the characteristic polynomial, under the action of Frobenius, of the (degree 2 component of the) \'etale cohomohology group for (a certain quotient of) the Fermat surface of degree $d$, whose image under blow-ups provides a smooth proper model for $E_d$ over $\mathbb{F}_q(t)$. Under the assumption that $d$ divides $p^n+1$ for some $n$, this characteristic polynomial is given precisely by the expression in the statement of the present proposition. 

The statement about the rank is then a straightforward consequence of the formula for the $L$-function 
of $E_d/\mathbb{F}_q(t)$.
\end{proof}



\begin{proposition}\label{prop:explicitformulaS_d}
Let $c_{\pm}(X)$ be defined as in \eqref{eq:cplusminus}. Then
\[
T_d(X) =-c_{\pm}(X) +\frac{\epsilon_d}{1-q^{-\frac 12}}+\sum_{\substack{e \mid d\\ e\nmid6}} \phi(e) \frac{ q^{ -(X \bmod o_{e}(q) ) /2 }}{ 1- q^{-o_{e}(q)/2}} +o_{X\rightarrow \infty}(1) 
\]
for $X$ large enough, where $0\leq (X \bmod \ell) \leq \ell-1$ is the remainder in the Euclidean division of
 $X$ by $\ell$. 
\end{proposition}

\begin{proof}
We combine Corollary \ref{thm:LimitDistForT} and Proposition \ref{thm:UlmerLfunction} to obtain
\begin{equation}\label{eq:SdAgain}
T_d(X) =
-c_{\pm}(X)+\frac{\epsilon_d}{1-q^{-\frac 12}}+\sum_{\substack{e \mid d \\ e\nmid 6}} \frac{\phi(e)}{o_{e}(q)}  \sum_{ k=0}^{o_{e}(q)-1} 
\frac{\ee^{2\pi\mathrm{i} k X /o_e(q)}}{1-q^{-1/2} \mathrm{e}^{-2\pi\mathrm{i}k / o_e(q)}}+o_{X\rightarrow\infty}(1).
\end{equation}
To simplify this, define
\[
f_{\ell}(X):= \sum_{ k=0}^{\ell-1} 
\frac{\ee^{2\pi \mathrm{i} k X /\ell}}{1-q^{-1/2} \mathrm{e}^{-2\pi i k / \ell }}
\]
for any positive integer $\ell$. Then, 
\begin{align*}
f_{\ell}(X) &=\sum_{ k=0}^{\ell-1} \ee^{2\pi \mathrm{i} k X /\ell} \sum_{m=0}^{\infty} q^{-m/2} \mathrm{e}^{-2\pi\mathrm{i} km / \ell }   = \sum_{m=0}^{\infty} q^{-m/2}\sum_{ k=0}^{\ell-1}   \ee^{2\pi i k (X - m) /\ell} \\
&=\sum_{m=0}^{\infty} q^{-m/2} \begin{cases}  \ell & \text{ if } X \equiv m \bmod \ell \\
0 & \text{ otherwise}
\end{cases} = \ell \sum_{j=0}^{\infty} q^{- ((X \bmod \ell)+j \ell) /2} = \frac{\ell q^{ -(X \bmod \ell ) /2 }}{ 1- q^{-\ell/2}}.
\end{align*}   
Now, we use this with $\ell=o_e(q)$ in \eqref{eq:SdAgain} and finish the proof.
\end{proof}

As a first consequence we deduce the existence and a nontrivial lower bound for $\delta(E_d)$ which is valid in general. We also get
 a conditional upper bound that will be used to prove Proposition~\ref{prop:notsobiased}.
It will be useful to consider the periodic part of $T_d(X)$, so we define
\begin{equation}
T_d^{\rm per}(X) :=-c_{\pm}(X) +\frac{\epsilon_d}{1-q^{-\frac 12}}+\sum_{\substack{e \mid d\\ e\nmid6}} \phi(e) \frac{ q^{ -(X \bmod o_{e}(q) ) /2 }}{ 1- q^{-o_{e}(q)/2}}.
\label{eq:T_dwithouterror} \end{equation}
(Note that $T_d^{\rm per}(X)=T_d(X)+o_{X\rightarrow \infty }(1)$.)

\begin{corollary} \label{corollary lower bound for delta}
The density $\delta(E_d)$ exists; more precisely assuming that $d\geq 7$ divides $p^n+1$ for some $n$, the function 
$T_d^{\rm per}(X)$ is $2n$-periodic. 
Furthermore, one has
$$ \delta(E_d) \geq \frac 1{2n}\,,$$
and, provided there exists some $X_0\geq 2$ such that $T_d^{\rm per}(X_0)<0$, one also has
$$
\delta(E_d) \leq 1-\frac 1{2n}\,.
$$
\end{corollary}

\begin{proof}
The existence of $\delta(E_d)$ is clear since $T_d(X)$ is periodic up to a function of size $o_{X\rightarrow\infty}(1)$.

Let us prove the second part of the statement. Let $k\geq 1$ be the integer such that $q=p^k$. Note first that the expression 
$$\sum_{\substack{e \mid d\\ e\nmid6}} \phi(e) \frac{ q^{ -(X \bmod o_{e}(q) ) /2 }}{ 1- q^{-o_{e}(q)/2}} $$
is $2n$-periodic. This follows since $p^{2n} \equiv 1 \bmod d$, which implies that $q^{2n} = p^{2kn} \equiv 1 \bmod e$ for every $e\mid d$, and hence $o_e(q) \mid 2n$.

If $X\equiv 0 \bmod 2n$, then Proposition \ref{prop:explicitformulaS_d} implies that

\begin{align*}
T_d(X) &=-c_{\pm}(X) +\frac{\epsilon_d}{1-q^{-\frac 12}}+\sum_{\substack{e \mid d\\ e\nmid6}} \phi(e) \frac{ 1 }{ 1- q^{-o_{e}(q)/2}} +o_{X\rightarrow \infty}(1), 
\end{align*}  
a quantity which is positive for $X$ large enough, since $\phi(d)$ is even for $d\geq 3$ and thus
$$\frac{ 2 }{ 1- q^{-o_{e}(q)/2}}  > c_{\pm}(X).  $$
Finally the upper bound is a trivial consequence of the $2n$-periodicity of $T_d^{\rm per}(X)$ and of the existence of some $X_0$ such that $T_d^{\rm per}(X_0)<0$, combined with the fact that $T_d(X)=T_d^{\rm per}(X)+o_{X\rightarrow\infty}(1)$.
\end{proof}

%
%
%

\subsection{Cases of extreme bias for Ulmer's family: proof of Theorem~\ref{th:extremebias}.}

\begin{proof}[Proof of Theorem \ref{th:extremebias}(i)]
If $d$ and $q$ satisfy one of the two stated assumptions, then $\epsilon_d \ge 1$. Then, the statement easily follows because
\begin{align*}
-c_{\pm}(X)+\frac{\epsilon_d}{1-q^{-1/2}} &\ge -\frac 1{1-q^{-1}} + \frac 1{1-q^{-\frac 12}} = \frac{q^{-\frac 12}-q^{-1}}{(1-q^{-1})(1-q^{-\frac 12})}>0,
\end{align*}
thus using Proposition~\ref{prop:explicitformulaS_d} we see that 
$T_d(X) > 0$ for all large enough $X$.
\end{proof}


\begin{proof}[Proof of Theorem \ref{th:extremebias}(ii)]
By Proposition \ref{prop:explicitformulaS_d} and by positivity, we have (recall that $p$ is large enough, and therefore so is $d$)
\begin{align*}
T_d(X) &= -c_{\pm}(X) +\frac{\epsilon_d}{1-q^{-\frac 12}}+\sum_{\substack{e \mid d\\ e\nmid6}} \phi(e) \frac{ q^{ -(X \bmod o_{e}(q) ) /2 }}{ 1- q^{-o_{e}(q)/2}} +o_{X\rightarrow \infty}(1) \\ &\geq \phi(d) q^{ -(X \bmod o_{d}(q) ) /2} -1-o_{p\rightarrow \infty}(1) -o_{X\rightarrow \infty}(1)  \\
& \geq \phi(d) q^{ -(o_{d}(q)-1 ) /2} -1-o_{p\rightarrow \infty}(1) -o_{X\rightarrow \infty}(1).
\end{align*}  
However, we have that $q^{2n/k} \equiv 1 \bmod d$, that is $o_d(q) \mid 2n/k$. We conclude that
\begin{align*}
T_d(X) &\geq  \phi(d) q^{\frac 12}  q^{-\frac nk} -1-o_{p\rightarrow \infty}(1) -o_{X\rightarrow \infty}(1)  \\
& =\phi(d) q^{\frac 12} (d-1)^{-1} -1-o_{p\rightarrow \infty}(1)-o_{X\rightarrow \infty}(1).
\end{align*}
This quantity is positive for large enough $X$, since $d$ is large enough, $\phi(d) /(d-1) \geq  (\ee^{-\gamma}+o(1)) / \log\log d $ and the condition on $n$ implies that $\log\log (p^n+1) \leq q^{\frac 12} /2 + \log\log p+1 $. 

\end{proof}



Our aim is now to prove Theorem~\ref{th:extremebias}(iii). As a preliminary result we give examples of very biased races, for which the lower bound of Corollary \ref{corollary lower bound for delta} is essentially sharp. We will then show using a result of Goldfeld that there exist arbitrarily large values of $p$ and $d$ satisfying the conditions of the statement (i) of the proposition.

\begin{proposition}\label{prop:highlybiased}
\begin{enumerate}
\item[(i)] Let $d\geq 7$ and $p\geq 3$ be two primes such that $p$ is a primitive root modulo $d$. Selecting $q=p^{\frac{d-1}2 +1}$, we have that $T_d(X)$ is quite biased towards negative values; precisely 
$$\frac{1}{d-1}\leq \delta(E_d)\leq \frac{4}{d-1}. $$
\item[(ii)] Assume that $d=p^n+1=2\ell$ with $\ell\geq 7$ a prime, $p\equiv 3 \bmod 4$ and $n\geq 4$ 
an even number. Pick $q=p^{n-1}$. Then $T_d(X)$ is quite biased towards negative values; precisely 
$$\frac 1{2n} \leq  \delta(E_d) \leq  \frac 2n. $$
\end{enumerate}
\end{proposition}

\begin{proof}
(i) We first see that the conditions of Ulmer's construction are satisfied. This is clear since $p^{\frac{d-1}2} \equiv -1 \bmod d$, hence $d\mid p^n+1$ with $n=(d-1)/2$.

Note also that since $d\geq 5$ is prime, $\epsilon_d=0$. Moreover, $o_d(q) \in \{(d-1)/2,d-1 \}$. This is clear since $q \equiv -p \bmod d$, and $p$ is a primitive root modulo $d$. Proposition \ref{prop:explicitformulaS_d} then takes the form
$$ T_d(X) = -c_{\pm}(X) + \phi(d) \frac{ q^{ -(X \bmod o_{d}(q) ) /2 }}{ 1- q^{-o_{d}(q)/2}}  +o_{X\rightarrow \infty} (1). $$


If $X \equiv j \bmod (d-1)$ with $j\notin \{ 0,1, (d-1)/2,(d-1)/2+1\}$, then
$$ T_d(X) \leq -q^{-\frac 12} +2 (d-1) q^{-2/2} +o_{X\rightarrow \infty} (1) =-q^{-\frac 12} \left(1- \frac{2(d-1)}{p^{\frac{d+1}4}} \right) +o_{X\rightarrow \infty} (1), $$
a quantity which is negative for $X$ large enough since for $d\geq 7$ and $p\geq 3$ we have 
$4(d-1)<p^{\frac{d+1}4}$. We have thus shown that $T_d(X)$ is negative for most of the values of $X\bmod (d-1)$.

Combining this with Corollary~\ref{corollary lower bound for delta} we conclude that 
$$ \frac 1{d-1} \leq \delta(E_d) \leq \frac 4{d-1}. $$

(ii) First note that the given choice of paramaters ensures that $\epsilon_d=0$. Proposition \ref{prop:explicitformulaS_d} yields the formula
\begin{align*}
T_d(X)= -c_{\pm}(X) + \phi(\ell) \frac{ q^{ -(X \bmod o_{\ell}(q) ) /2 }}{ 1- q^{-o_{\ell}(q)/2}}  + \phi(2\ell) \frac{ q^{ -(X \bmod o_{2\ell}(q) ) /2 }}{ 1- q^{-o_{2\ell}(q)/2}} +o_{X\rightarrow \infty} (1).
\end{align*}  

We have $q^n \equiv (-p^{-1})^n \equiv -1 \bmod d $. We claim that $n$ is the least positive integer such that this congruence holds. Indeed this minimality condition holds by definition for the congruence $p^n \equiv -1 \bmod d$. Now 
$q \equiv -p^{-1}\,\bmod\, d$ and $n$ is even, thus the claim follows. Since $(\mathbb Z/d\mathbb Z)^\times$ is cyclic (it is isomorphic to $(\mathbb Z/\ell\mathbb Z)^\times$), this implies 
that $o_d(q)=o_{2\ell}(q)=2n$. 
Now consider an integer $j$ such that $4 \leq j  \leq 2n-1$. We have that
$$ \phi(2\ell) \frac{ q^{ -(j \bmod o_{2\ell}(q) ) /2 }}{ 1- q^{-o_{2\ell}(q)/2}} \ll \ell q^{ -4/2 } \ll p^{n- 2(n-1)}= p^{2-n}.$$
Hence, since $-c_{\pm}(j) < 0$ and $c_{\pm}(j) \gg q^{-\frac 12}=p^{(1-n)/2}$, we have for $p$ and $X$ large enough and $n\geq 4$ that $T_d(X)<0$ as soon as $X\equiv j \bmod 2n$. Hence,
$$ 0 \leq  \delta(E_d) \leq \frac 2n.  $$
As for the lower bound, it is given once more by Corollary~\ref{corollary lower bound for delta}. 

\end{proof}

We will now see that there are infinitely many choices of primes $p$ and $d$ in (i) of the proposition (in which both can be arbitrarily large) for which $p$ is a primitive root modulo $d$. This is a corollary of the following result of Goldfeld.

\begin{theorem}[Goldfeld \cite{Go}]
\label{lemma Goldfeld}
Let $x\geq 2$ be a real number and let $Li(x):=\int_2^x{(\log t)^{-1}{\rm d} t}$.
Let $1<A\leq x$. Then for each $D\geq 1$, 
\begin{equation}
 N_a(x) := \# \{ p\leq x : a \text{ is a primitive root} \bmod p \} = c Li(x) +O_D\left( \frac x{(\log x)^D}\right), 
 \label{equation Artin}
\end{equation}
for all $a\leq A$ with at most $ c(D) A^{\frac 9{10}} (5 \log x +1)^{D+2+\frac{\log x}{\log A}} $ exceptions, where $c$ is Artin's Constant:
$$ c:= \prod_p \left( 1-\frac 1{p(p-1)} \right). $$
\end{theorem}

\begin{corollary}
\label{cor:Goldfeld}
We have for any fixed $C\geq 1$ the estimate
$$ \sum_{\substack{ p,d \leq x \text{ both prime} \\ p \text{ is a primitive root} \bmod d }}1 = c Li(x)^2 \left(1+ O_C\left(\frac 1{(\log x)^C} \right) \right), $$
and as a consequence,
$$\# \{ p,d \in (x,2x] \text{ both prime} : p \text{ is a primitive root} \bmod d\} \sim \frac{3c x^2}{(\log x)^2}. $$
\end{corollary}

\begin{proof}

Pick $A=x$ and $D=C$ in Goldfeld's Theorem, and let $\mathcal A(x)$ be the set of $a\leq x$ for which \eqref{equation Artin} does not hold, thus $|\mathcal A(x)| \ll_C x^{\frac 9{10}} (\log x)^{C+3}$. We then write
\begin{align*}
\sum_{\substack{ p,d \leq x \text{ both prime} \\ p \text{ is a primitive root} \bmod d }}1 = \sum_{p\leq x} N_p(x) &= \sum_{p\leq x} c Li(x) + \sum_{p\leq x} (N_p(x) - cLi(x)) \\
& = c Li(x)^2 +O\left(   \sum_{p\leq x} |N_p(x) - cLi(x)|\right) +  O\left( \frac {x^2}{(\log x)^C}\right) ,
\end{align*}
by the Prime Number Theorem.  The first error term is bounded as follows:
\begin{align*}
\sum_{p\leq x} |N_p(x) - cLi(x)|  &= \sum_{ \substack{p\leq x \\ p \in \mathcal A(x)}} |N_p(x) - cLi(x)|  +\sum_{ \substack{p\leq x \\ p \notin \mathcal A(x)}} |N_p(x) - cLi(x)| \\
& \ll_C |\mathcal A (x)| \pi(x) + \frac{x^2}{(\log x)^{C+1}} \ll_C \frac {x^2}{(\log x)^{C}}.
\end{align*} 

\end{proof}

\begin{proof}[Proof of Theorem \ref{th:extremebias}(iii)]
This is a direct consequence of Proposition \ref{prop:highlybiased}(i) and Corollary \ref{cor:Goldfeld}. Note that the curves $E_d$ in Proposition \ref{prop:highlybiased} have rank either one or two (depending on the parity of $(d-1)/2$). Indeed the proof of Proposition~\ref{prop:highlybiased} shows that the multiplicative order of $q$ modulo $d$ is 
$d-1$ (resp. $(d-1)/2$) if $(d-1)/2$ is even (resp. odd), and we get the corresponding value of 
the rank by applying Proposition~\ref{thm:UlmerLfunction}.
\end{proof}

%
%

\subsection{Cases of moderate bias for Ulmer's family: proof of Theorem~\ref{th:moderatebias}.}

We start this section by proving a first result about moderate bias. We will later specialize in order to obtain races which are not biased at all.

\begin{proposition}
\label{prop:notsobiased}
Fix $n\geq 2$ even, $p\equiv 3 \bmod 4$ and $k\geq 1$. Pick $q=p^{kn+1}$ and $d=p^n+1$. Then for $n$ fixed and $p$ large enough, $T_d(X)$ is moderately biased, that is 
$$\frac 1{2n} \leq  \delta(E_d) \leq 1- \frac 1{2n}. $$
\end{proposition}

\begin{remark}
Note that the bounds of Proposition \ref{prop:highlybiased}(ii) are much more precise. However, although we believe that there are infinitely many curves satisfying the hypotheses of Proposition \ref{prop:highlybiased}(ii), this seems hard to prove given the restrictive arithmetic conditions that $d$ has to satisfy.
\end{remark}

\begin{proof}[Proof of Proposition \ref{prop:notsobiased}]
Since the given choice of paramaters ensures that $\epsilon_d=0$, Proposition \ref{prop:explicitformulaS_d} takes the form
\begin{align*}
T_d(X)= -c_{\pm}(X) + \sum_{\substack{ e\mid d \\ e\nmid 6}} \phi(e) \frac{ q^{ -(X \bmod o_{e}(q) ) /2 }}{ 1- q^{-o_{e}(q)/2}}  +o_{X\rightarrow \infty} (1).
\end{align*}  

We now show that for each $e\mid d$ with $e\nmid 6$, we have $o_e(q) \geq 3$. Note first that $(p-1,e)$ and $(p+1,e)$ both divide $2$. Indeed we have obviously $ (p-1,e) \mid (p-1,d)$ and $(p+1,e) \mid (p+1,d)$. We compute
\begin{align*} (p-1,d)& = (p-1, p^n+1) = (p-1,p^{n-1}+1) = \dots = (p-1,p+1) = 2\,;  \\
(p+1,d) &= (p+1, p^n+1) = (p+1,p^{n-1}-1) =(p+1,p^{n-2}+1) = \dots = (p+1,p-1) = 2\,,  
\end{align*}
since $n$ is even. Note that $e\nmid 6$, and also $d \equiv 2 \bmod 4$ implies $e\neq 4$. We conclude from the above computation that $e \nmid (p+1)(p-1)$, that is $o_e(p) \geq 3$. Since $q\equiv \pm p \bmod e$ and $p\not \equiv \pm 1 \bmod e$, we also have that $o_e(q)\geq 3$.

Using this fact, we have that if $X \equiv 2 \bmod 2n$, then
\begin{align*}
T_d(X)= -c_{\pm}(X) + \sum_{\substack{ e\mid d \\ e\nmid 6}} \phi(e) \frac{ q^{ -2 /2 }}{ 1- q^{-o_{e}(q)/2}}  +o_{X\rightarrow \infty} (1).
\end{align*}  

This last quantity is negative for large enough $X \equiv 2 \bmod 2n$, since $c_{\pm}(X) = 1+o_{p \rightarrow \infty }(1)$, and 
$$\sum_{\substack{ e\mid d \\ e\nmid 6}} \phi(e) \frac{ q^{ -2 /2 }}{ 1- q^{-o_{e}(q)/2}}  \ll q^{-1} d \ll p^{n-(kn+1)} \ll p^{-1}, $$
which is negligible compared to $c_{\pm}(X)$.
We conclude by invoking Corollary~\ref{corollary lower bound for delta}.


\end{proof}

\begin{proof}[Proof of Theorem \ref{th:moderatebias}]

Firstly, $d \mid p^2+1 \equiv 2 \bmod 3$ if $p\neq 3$ (otherwise $p^2+1 \equiv 1 \bmod 3$) and $q \equiv p \equiv 3 \bmod 4$, hence $\epsilon_d=0$. Note also that if $e \mid d$ with $e\notin \{1,2\}$, then $q^2 \equiv p^2 \equiv -1 \bmod e$, hence $o_e(q)=4$. Proposition \ref{prop:explicitformulaS_d} becomes 
\begin{align*}
T_d(X)= -c_{\pm}(X) + \sum_{\substack{ e\mid d \\ e\neq 1,2}} \phi(e) \frac{ q^{ -(X \bmod 4 ) /2 }}{ 1- q^{-2}}  +o_{X\rightarrow \infty} (1),
\end{align*} 
hence $T_d^{\rm per}(X)= -c_{\pm}(X) + \sum_{ e\mid d,\, e\neq 1,2} \phi(e) q^{ -(X \bmod 4 ) /2 }(1- q^{-2})^{-1}$ is $4$-periodic. 

If $X \equiv 0 \bmod 4$, then
$$ T_d(X) = -\frac q{q-1}+ \sum_{\substack{ e\mid d \\ e \neq 1,2}} \frac{\phi(e)}{1-q^{-2}} +o_{X\rightarrow \infty} (1) \geq -2+\frac{d-2}{1-q^{-2}}+o_{X\rightarrow \infty} (1),  $$
which is positive for $X$ large enough. 

If $X\equiv 1 \bmod 4$, then
$$ T_d(X) = -\frac {q^{\frac 12}}{q-1}+ \sum_{\substack{ e\mid d \\ e \neq 1,2}} \frac{\phi(e) q^{-\frac 12}}{1-q^{-2}} +o_{X\rightarrow \infty} (1) \geq q^{-\frac 12}\left(\frac{d-2}{1-q^{-2}} - 2 \right)+o_{X\rightarrow \infty} (1),  $$
which is again positive for $X$ large enough.

As for $X\equiv 2 \bmod 4$, we have 
$$ T_d(X) = -\frac q{q-1}+ \sum_{\substack{ e\mid d \\ e \neq 1,2}} \frac{\phi(e) q^{-1}}{1-q^{-2}} +o_{X\rightarrow \infty} (1) \leq -1 +2dq^{-1}+o_{X\rightarrow \infty} (1),  $$
which is negative for $X$ large enough since $2d \leq 2(p^2+1) < p^5/2\leq q/2$.

Finally, for $X\equiv 3 \bmod 4$ we have 
$$ T_d(X) = -\frac {q^{\frac 12}}{q-1}+ \sum_{\substack{ e\mid d \\ e \neq 1,2}} \frac{\phi(e) q^{-{\frac 32}}}{1-q^{-2}} +o_{X\rightarrow \infty} (1) \leq q^{-\frac 12}\left(-1 +2dq^{-1}\right)+o_{X\rightarrow \infty} (1),  $$
which is negative for $X$ large enough. Since $T_d(X)$ is positive for asymptotically half of the values of $X$, we conclude that $\delta(E_d)=\frac 12$.
\end{proof}

\subsection{On the closure of the set of densities $\delta(E_d)$: proof of 
Theorem~\ref{thm:limitingpoints}}

We first need the following preliminary result.
\begin{proposition}\label{prop:limitpoints}
Let $p\geq 17$ and $n\geq 3$ be primes such that $n \nmid p+1$, and set $d=(p^n+1)/(p+1)$ and $q=p^k$ with $k\geq 4$ even and coprime to $n$. Then the curve $E_d$ has rank exactly $(d-1)/n$, and we have that 
$$ \frac 1k -\frac 2{nk} \leq \delta(E_d) \leq \frac 1k + \frac 1{2n}\,. $$
\label{prop:1/kaccumulation}
\end{proposition}

\begin{proof}

We first see that $\epsilon_d=0$. Indeed we have that
\begin{equation}\label{eq:dmod2}
 d = 1-p+p^2-p^3+\dots +p^{n-1} \equiv n \equiv 1 \bmod 2.  
 \end{equation}
Moreover, if $p\equiv 1 \bmod 3$, then $d\equiv 2/2 \equiv 1 \bmod 3$. If $p\equiv 2 \bmod 3,$ then we have
$$ d = 1-p+p^2-p^3+\dots +p^{n-1} \equiv 1-2+1-2+\dots +1 \equiv n \not \equiv 0 \bmod 3,   $$
hence $6\nmid d$ and $\epsilon_d =0$.

We will now show that for every $e\mid d$ with $e\nmid 6$, $o_e(q)=n$. First, $p^{n} \equiv -1 \bmod e$, so $o_e(p) \in \{ 2,2n \}$ (indeed $-1\not\equiv 1\bmod e$ since $e\nmid 6$). Note that $(e,p+1)=1$; indeed
\begin{align*}
(d,p+1)&=(p^{n-1}-p^{n-2}+p^{n-3}-\dots+1,p+1)= (-2 p^{n-2}+p^{n-3}-\dots +1)\\
 &= \dots = (-(n-1)p+1,p+1) =(n,p+1)=1,
\end{align*} 
since $n$ is prime and $n\nmid p+1$. We also have 
$$ (p^n+1,p-1)=(p^{n-1}+1,p-1)=\dots =(p+1,p-1)=2, $$
hence $(d,p-1)=1$ (since $d$ is odd by~\eqref{eq:dmod2}). We conclude that $(e,(p+1)(p-1))=1$, and so since $e>1$ (recall that $e\nmid 6$), $p^2 \not \equiv 1 \bmod e$ and thus $o_e(p)=2n$. Since $(k,2n)=2$, it follows that $o_e(q)=n$. 

We now turn to the study of $T_d(X)$. We have
\begin{align*}
T_d(X) &= -c_{\pm}(X) + \sum_{\substack{ e \mid d \\ e\nmid 6}} \phi(e) \frac{q^{-( X\bmod n)/2}}{1-q^{-n/2}}+o_{X\rightarrow \infty}(1)\\ &= -c_{\pm}(X) + (d-1)\frac{q^{-( X\bmod n)/2}}{1-q^{-n/2}}+o_{X\rightarrow \infty}(1). 
\end{align*}  

If $X \equiv j \bmod 2n$ with $0\leq j \leq  2(n-2)/k$, then
$$T_d(X) \geq -2+ (d-1) q^{-\frac{n-2}k }+o_{X\rightarrow \infty}(1) \geq -2 + \frac p2 +o_{X\rightarrow \infty}(1), $$
which is positive for large enough $X$.

If $X \equiv j \bmod 2n$ with $ (2n-1)/k+1 \leq j \leq 2n-1$, then
$$T_d(X) \leq -\frac{q^{-\frac 12}}{1-q^{-1}} + 2(d-1) q^{-\frac{n-1/2}k-\frac 12 }+o_{X\rightarrow \infty}(1) \leq -\frac 1{q^{\frac 12}} + \frac 4{(pq)^{\frac 12 }} +o_{X\rightarrow \infty}(1), $$
which is negative for large enough $X$, since we have assumed $p\geq 17$.

We conclude that 
$$ \frac 1k - \frac{2}{nk} \leq  \delta(E_d) \leq \frac 1k + \frac 1{2n} - \frac 1{2nk}.  $$
%
%
%
\end{proof}

\begin{proof}[Proof of Theorem \ref{thm:limitingpoints}]
From Theorems \ref{th:extremebias} and \ref{th:moderatebias}, we already know that $\{0,\frac 12,1\} \subset \overline{S}$.

Fix $m\geq 2$ and $\epsilon>0$. Let $p>\min(17, 2\epsilon^{-1})$, and pick a prime $n>p+1$, so that $n\nmid p+1$. Letting $k=2m$ in Proposition \ref{prop:1/kaccumulation}, we find a curve $E_d$ such that
$$ \frac 1{2m} -\epsilon \leq \delta(E_d) \leq \frac 1{2m} + \epsilon. $$
The theorem follows since $\epsilon$ is arbitrary.
 \end{proof}
 
  The proof of Proposition~\ref{prop:limitpoints} shows in fact that under the stated assumptions, $2n$ is the smallest period of the periodic part $T_d^{\rm per}(X)$ of $T_d(X)$. 

\section{Central Limit Theorem} \label{sec:CLT}
The goal of this section is to prove a Central Limit Theorem, in particular we will prove Theorems \ref{FirstTheorem}(ii) and \ref{thm:generic bias}.

\begin{definition}\label{LIDefinition}
Recall from \eqref{eq:InverseZeroForAllM} and \eqref{CaseMEqual1} that we denote the inverse zeros of $L(E/K, T)$ by $\{\gamma_{j} = q\ee^{\ii \theta_{j}}\}$ with $j=1, \dots, N_{E/K}$. Let $\epsilon(E/K) = \pm 1$ be the sign of the functional equation for $E/K$, that is, the unique number satisfying
\begin{equation}\label{eq:funceq}
L(E/K, T) = \epsilon(E/K) (qT)^{N_{E/K}} L(E/K, 1/(q^2T)).
\end{equation}
(See e.g.~\cite[Th.~$2.2.1$]{Ulm11}.) Define the set of \emph{forced zeros} of $E/K$ to be
\begin{equation}\label{ForcedZeros}
\mathrm{FZ}(E/K): =
\begin{cases}
\{ \epsilon(E/K) q \} & \text{ if $N_{E/K}$ is odd,} \\
\{ q, -q \} & \text{ if $N_{E/K}$ is even and }\epsilon(E/K) = -1, \\
\{\} & \text{ otherwise. } 
\end{cases}
\end{equation}
We will denote by $\{\theta_1 , ..., \theta_k\}$ the set of angles which do not come from forced zeros of $L(E/K,T)$. We will say that an elliptic curve $E$ over $K$ satisfies the \emph{linear independence} (LI) hypothesis if the multiset
\[
\big\{ \theta_{j}/\pi \colon 0 < \theta_{j} \le \pi,\, j\in\{1,\ldots k\}\big\} \cup \{ 1 \}
\]
is linearly independent over $\mathbb{Q}$. 
\end{definition}
The idea behind this definition is quite simple. 
When we say that $E/K$ satisfies LI, we are willing to ignore any trivial multiplicative relation among the zeros, namely, the relations coming from complex conjugation and forced zeros $\gamma = \pm q$, as well as those $\gamma = q$ arising from the vanishing of $L(E/K, T)$ at the central point, or, a positive analytic rank of $E/K$. 
In our work~\cite{CFJ} we prove a stronger statement involving 
the possible multiplicative relations among inverse zeros of \emph{reduced} $L$-functions i.e.~the $\Qq$-polynomial obtained by quotienting $L(E/K,T)$ by the product of linear factors corresponding to $\mathrm{FZ}(E/K)$. (Note that any forced zero $\gamma = \pm q$ would cause linear dependence over $\mathbb{Q}$.) Indeed, our work \cite{CFJ} shows that generically not only does LI hold, but also the rank is at most one.

Before we prove a Central Limit Theorem which implies Theorem \ref{FirstTheorem}(ii), we need several lemmas.


\begin{lemma}\label{Littlewood}
We have that
\[
\frac{L'}{L}(E/K, q^{-3/2}) = O_C(\log_q N_{E/K}).
\]
\end{lemma}

\begin{proof}
This is a function field analogue of Littlewood's bound $(L'/L)(1, \chi) = O(\log\log q)$ for Dirichlet characters $\chi$ modulo $q$. 
Let $\mathcal{D}$ be an effective divisor on the curve $C$ (recall that $C/\mathbb{F}_q$ is the curve whose function field is by definition $K$) i.e.~$\mathcal{D} = \sum_v n_v\cdot v$ where $n_v$ are nonnegative integers, $n_v=0$ for all but finitely many $v$'s, and $v$ runs over the set of places of $K$. Using the standard definition of the degree of a divisor (for $\mathcal D$ as above, $\deg \mathcal D=\sum_v n_v\deg v$), we let $|\mathcal{D}| := q^{\deg(\mathcal{D})}$. Also, we define the \emph{von Mangoldt function on} $C$ by
\[
\Lambda_C(\mathcal{D}) = \begin{cases} \deg v & \text{ if $\mathcal{D} = n_v \cdot v$ for some $v$,} \\ 0 & \text{ otherwise.} \end{cases}
\]
Define 
\[
a_{n_v\cdot v}(E) = {\alpha_v}^{n_v} + {\beta_v}^{n_v}\,,
\]
with notation as in \S\ref{sec:Notations}, and further $\alpha_v=a_v(E)$, $\beta_v=0$ if $v$ is a place of bad reduction 
of $E/K$. We claim that:
\begin{equation}\label{eq:LogDevLambda}
-\frac{L'}{L}(E/K, q^{-3/2})=\sum_{\substack{\mathcal D\colon \mathcal D=n_v\cdot v\\ n_v\geq 0,\, v\text{ place of }K} }\frac{\Lambda_C(\mathcal D)a_{\mathcal D}(E)}{|\mathcal{D}|^{3/2}}\,.
\end{equation}
To see this, we use~\eqref{LocalGood},~\eqref{LocalBad} which coincide in the case $m=1$ (which is the relevant case) and an exact version of their consequence~\eqref{LogDerivativeLocal}. 

We rewrite the right hand side of~\eqref{eq:LogDevLambda} as follows
\begin{equation}\label{eq:Sumnv}
\sum_{n\geq 1}\sum_{v\text{ place of }K}\frac{\deg v}{q^{n\deg v}}\cdot\frac{a_{n\cdot v}(E)}{q^{(n\deg v)/2}}\,.
\end{equation}
The sum of the terms with $n\geq 2$ is 
$$
\ll \sum_{n\geq 2} \sum_{v\text{ place of }K}\frac{\deg v}{q^{n\deg v}}= \sum_{d\geq 1}\sum_{\substack{v\text{ place of }K\\ \deg v=d}}\frac{d}{q^{2d}}\cdot\frac{1}{1-q^{-d}}\ll_C\sum_{d\geq 1}q^{-d}=\frac{1}{q-1}\,,
$$
by~\eqref{PrimeNumberTheorem}. Next we split the rest of the sum~\eqref{eq:Sumnv} as $S_1+S_2$ where
$$
S_1:=\sum_{d\leq 2\log_qN_{E/K}}\frac{d}{q^d}\sum_{\substack{v\text{ place of }K\\ \deg v=d}}\frac{a_v(E)}{q^{\deg v/2}}
\,,\qquad S_2:=\sum_{d> 2\log_qN_{E/K}}\frac{d}{q^d}\sum_{\substack{v\text{ place of }K\\ \deg v=d}}\frac{a_v(E)}{q^{\deg v/2}}\,.
$$
Applying again~\eqref{PrimeNumberTheorem} we trivially have
$$
S_1\ll_C \log_q N_{E/K}\,.
$$
Finally we deduce $S_2\ll N_{E/K}/{q^{\log_q N_{E/K}}}=1$ from Theorem~\ref{thm:WienerIkehara}.
\end{proof}

The following sum over the angles $\theta_j$ which do not come from forced zeros gives an estimate for the variance of the random variable $X_E$. 

\begin{lemma}
\label{lem:sumoverzeros}
We have the estimate
\[
I_E:=  \sum_{j=1}^k \left|\frac{2}{1 - q^{-\frac 12}\ee^{-\ii\theta_j}} \right|^2
= \frac{2q}{q-1} N_{E/K} +O\left(\log N_{E/K}+\rank(E/K) \right).
\]
\end{lemma}
\begin{proof}
The quantity $I_E$
can be computed by evaluating the log derivative of \eqref{LfunctionUnderLI} at $T=q^{-1/2}$ (see Page 1370 in \cite{Cha08} for a similar calculation). We see that 
\begin{equation}\label{IEbeforeFunctionalEquation}
I_E = \frac{4q}{q-1}
\left(
q^{-1/2}\frac{L'}{L}(E/K, q^{-1/2}) - \frac{\sqrt q}{\sqrt q - 1}\rank(E/K) - k
\right)
.
\end{equation}
The functional equation~\eqref{eq:funceq} yields
\begin{equation}\label{FE}
q^{-1/2}\frac{L'}{L}(E/K, q^{-1/2}) = N_{E/K} - q^{-3/2}\frac{L'}{L}(E/K, q^{-3/2}).
\end{equation}
This gives in turn
\begin{equation}\label{FunctionalEquation}
I_E = 
\frac{2q}{q-1}N_{E/K} + \frac{4q}{q-1}\left(\frac12 - \frac{\sqrt q}{\sqrt q - 1}\right)r(E/K)
-
\frac{4\sqrt q}{q-1}
\frac{L'}{L}(E/K, q^{-3/2}).
\end{equation}
The desired estimate follows from an application of Lemma \ref{Littlewood}.
\end{proof}


The following lemma is an application of the Berry-Esseen inequality.

\begin{lemma}
\label{lem:berryesseen}
Fix two parameters $0<\epsilon<1$ and $1\leq M \leq  \epsilon^{-\frac 12}$, and assume that the characteristic function $\widehat X(\xi)$ of the random variable $X$ satisfies the following properties:
\begin{enumerate}
\item[(i)] $ |\widehat X(\xi)| \leq 10 \xi^{-4} $ for $|\xi|\geq \epsilon^{-\frac 14}$, and
\item[(ii)] $|\log \widehat X(\xi) + \frac{\xi^2}2| \leq  10 \epsilon (M\xi^2+\xi^4) $ for $|\xi| \leq \epsilon^{-\frac 14}$.
\end{enumerate}
Then, the distribution function $F_X(x)$ of $X$ satisfies
$$ \sup_{x\in \mathbb R}| F_X(x) - G(x)| \ll M\epsilon, $$
where $G(x)$ is the distribution function of the Gaussian and the implied constant is absolute. (The constant $10$ appearing both in {\rm (i)} and {\rm (ii)} can be replaced by any positive absolute constant.)
\end{lemma}
\begin{proof}
We will apply the Berry-Esseen inequality \cite[Theorem 2a]{Ess45}
\begin{equation}
\sup_{x \in \mathbb R}|F_X(x)-G(x)|  \ll \int_{-T}^T \frac{\hat X(\xi)-\ee^{-\frac{\xi^2}2}}{\xi}\, \dd\xi + \frac 1T.
\label{equation Berry-Esseen}
\end{equation} 
Taking $T=\epsilon^{-1}$ and applying the hypotheses on $\widehat X(\xi)$, we note that the integral equals
\begin{align*} &\int_{|\xi|\leq \epsilon^{-\frac 14}} \ee^{-\xi^2/2} \frac{e^{O(\epsilon (M\xi^2+\xi^4))} -1}{\xi} \, \dd\xi +  O\left(\int_{\epsilon^{-\frac 14}\leq |\xi|\leq \epsilon^{-1}} 
\left|\frac{\xi^{-4}}{\xi}\right| \, \dd\xi\right) \\ &\ll 
\int_{|\xi|\leq \epsilon^{-\frac 14}} \ee^{-\xi^2/2} \frac{\epsilon (M\xi^2+\xi^4)}{|\xi|} \, \dd\xi + \epsilon  \\
&\leq  \epsilon \int_{\mathbb R} |M\xi+\xi^3| \ee^{-\xi^2/2} \, \dd\xi +\epsilon \ll M\epsilon. 
\end{align*}
The result follows.

\end{proof}

\begin{theorem}\label{CentralLimitTheorem}
Suppose that $\{ E/K \}$ is a family of elliptic curves satisfying LI 
such that
\begin{equation}\label{BoundedRank}
\frac{\rank(E/K)}{\sqrt{N_{E/K}}} \to 0
\end{equation}
as $N_{E/K}\to\infty$. Let $X_E$ be the random variable associated to the limiting distribution of $T_E(X)$ (see 
Corollary~\ref{thm:LimitingDistribution} in the case $V=-U_1$), and define the normalized random variable $Y_E:=\sqrt{\frac{q-1}{q}}
X_E/\sqrt{N_{E/K}}$. Then, $Y_E$ converges in distribution to the standard Gaussian. 
More precisely, the distribution function $F_E$ of $Y_E$ satisfies
$$ \sup_{x\in\mathbb R} |F_E(x) - G(x)|\ll \frac {\rank(E/K)+1}{\sqrt{N_{E/K}}}, $$
where $G$ denotes the distribution function of the Gaussian. 
\end{theorem}


\begin{proof}
Define $m(E/K, -q)$ to be the multiplicity of the inverse zero $\gamma = -q$ in $L(E/K, T)$. Then, we can write 
\begin{equation}\label{LfunctionUnderLI}
L(E/K, T) = (1+qT)^{m(E/K, -q)}(1 - qT)^{\rank(E/K)}\prod_{j=1}^k
\left[ (1-\gamma_jT)(1-\overline{\gamma}_jT) \right],
\end{equation}
with $\gamma_j = q \ee^{\ii \theta_j}$ and $0< \theta_j < \pi$ for $j = 1, \dots, k$. We have

\begin{equation}
N_{E/K}= m(E/K, -q) +\rank(E/K) + 2k.
\end{equation}
From \eqref{LfunctionUnderLI} and Corollary \ref{thm:LimitDistForT}, it is easy to deduce that
\begin{multline}\label{TXmodified}
T_E(X) = 
\frac{\sqrt q}{\sqrt q - 1}
\mathrm{rank}(E/K) - c_{\pm}(X) \\
-(-1)^X m(E/K, -q) \frac{\sqrt q}{\sqrt q + 1}
+
2 \sum_{j = 1}^k \Re\left(\frac{\ee^{\ii \theta_jX}}{1 - q^{-\frac 12}\ee^{-\ii \theta_j}} \right) + o_{X\to\infty}(1).
\end{multline}

From now on, we will assume that $E/K$ satisfies LI.
This implies in particular that
\begin{equation*}
m(E/K, -q)=
\begin{cases}
1 & \text{ if  $-q\in\mathrm{FZ}(E/K)$}, \\
0 & \text{ if  $-q\not\in\mathrm{FZ}(E/K)$}. \\
\end{cases}
\end{equation*}
Define $\mathcal{B}_E(\xi)$ by
\begin{equation}\label{FourierMean}
\mathcal{B}_E(\xi): = 
\frac12
\left[
\exp(\ii B_0 \,\xi) + \exp(\ii B_1 \, \xi)
\right]
\end{equation}
with
\[
B_0:= \frac{\sqrt q}{\sqrt q - 1}\mathrm{rank}(E/K) - \frac{q}{q - 1} + 
m(E/K, -q) \frac{\sqrt q}{\sqrt q + 1}
\]
and
\[
B_1:=\frac{\sqrt q}{\sqrt q - 1}\mathrm{rank}(E/K) - \frac{\sqrt q}{q - 1} -
m(E/K, -q) \frac{\sqrt q}{\sqrt q + 1}.
\]

To prove Theorem~\ref{CentralLimitTheorem}, we proceed as in~ \cite[Th.~$6.2$]{Cha08}. First, we find from the expression \eqref{TXmodified} that
\begin{equation}\label{FourierTransform}
\widehat X_E(\xi) = \mathcal{B}_E(\xi) \prod_{j=1}^k
J_0 \left(  \frac{2 \xi}{|1- q^{-\frac 12} \ee^{-\ii\theta_j}|}\right).
\end{equation}
Indeed, note that \eqref{TXmodified} implies that $X_E=X_1 +X_2$, where $X_1$ is a Bernoulli random variable with values $B_0$ and $B_1$, and $X_2:=2 \sum_{j = 1}^k \Re(Z_{j})/|1 -  q^{-\frac 12} \ee^{-\ii\theta_j}|$ with $Z_j$ uniformly distributed on the unit circle in $\mathbb C$. LI
then implies that the random variables $X_1,Z_1,Z_2, \dots , Z_k$ are mutually independent, and therefore the characteristic function of $X_E$ is given by the product of the characteristic functions of these random variables. We then obtain \eqref{FourierTransform} by performing a standard calculation, noting that $\mathbb E[\ee^{\ii tX_1}] = \frac 12 \ee^{\ii tB_0}+\frac 12 \ee^{\ii tB_1}$, and $\mathbb E[\ee^{\ii t\Re(Z_j)}]=J_0(t). $

It follows that the normalized random variable $Y_E$ satisfies
\begin{align*}
\log \widehat Y_E(\xi) = \log\mathcal{B}_E\left(\sqrt{\frac{q-1}{q}} \frac{\xi}{\sqrt{N_{E/K}}} \right)
 + 
\sum_{j=1}^k \log J_0
\left(
\left| \frac{2}{1 - q^{-\frac 12}\ee^{-\ii\theta_j}}\right| \sqrt{\frac{q-1}{q}} \frac{\xi}{\sqrt{N_{E/K}}}
\right).
\end{align*}
The Taylor series of $\log J_0(z)$ has radius of convergence slighly larger than $12/5$, since $J_0(z)$ has no zero in this region. Also, the argument of $\log J_0$ in the last equation never exceeds $6.83 |\xi|/\sqrt{N_{E/K}}$ in absolute value, and hence we have for $|\xi|\leq .35\sqrt{N_{E/K}}$ that
$$ \log \widehat Y_E(\xi) = O\left(|\xi|\frac{\rank(E/K)}{\sqrt{N_{E/K}}}\right) -\frac{q-1}{4q} \frac{\xi^2}{N_{E/K}}
\sum_{j=1}^k \left|\frac{2}{1 - q^{-\frac 12}\ee^{-\ii\theta_j}} \right|^2
+
O\left(\frac{\xi^4}{N_{E/K}}\right).
$$

Given the assumptions of the theorem and Lemma \ref{lem:sumoverzeros}, it follows that $\log \widehat Y_E(\xi) \rightarrow -\tfrac{\xi^2}2$ pointwise as $N_{E/K}$ tends to infinity. In light of Levy's theorem, this establishes the first part of the theorem.


For the second part, we write $Y_i:=(1-q^{-1})^{-\frac 12} X_i/\sqrt{N_{E/K}}$ ($i=1,2$), where the $X_i$ were defined earlier in the proof. We will apply Lemma \ref{lem:berryesseen} to the random variable $Y_2$, with the parameters $\epsilon := 1/N_{E/K}$, $M=c(\log N_{E/K}+\rank(E/K))$, where $c$ is the implied constant in Lemma \ref{lem:sumoverzeros}.
With the help of Lemma \ref{lem:sumoverzeros}, we see that the characteristic function of $Y_2$ 
\[
\widehat Y_2(\xi)= \prod_{j=1}^k  J_0
\left(
\left| \frac{2}{1 - q^{-\frac 12}\ee^{-\ii\theta_j}}\right| \sqrt{\frac{q-1}{q}} \frac{\xi}{\sqrt{N_{E/K}}}
\right)
\]
satisfies the estimate
$$\log \widehat Y_2(\xi)= -\frac{\xi^2}2 +O\left(\frac{\log N_{E/K} + \rank(E/K)}{N_{E/K}}\xi^2 + \frac{\xi^4}{N_{E/K}}\right) $$ in the range $|\xi|\leq 0.35 \sqrt{N_{E/K}}$; hence (ii) of Lemma \ref{lem:berryesseen} holds for $N_{E/K}$ large enough. Moreover, the bound $\log J_0(x) \leq -x^2/2$ valid for $|x|\leq 12/5$ combined with Lemma \ref{lem:sumoverzeros} imply that 
$$\log \widehat Y_2(\xi)  \leq -(1+o(1))\frac{\xi^2}2 $$
for $|\xi|\leq 0.35 \sqrt{N_{E/K}}$, and hence we clearly have $\widehat Y_2(\xi) \leq \xi^{-4}$ in this range. To show that (i) of Lemma \ref{lem:berryesseen} holds, it remains to show that $|\widehat Y_2(\xi)| \ll \xi^{-4} $ in the range $|\xi|\geq 0.35 \sqrt{N_{E/K}}$.

In the range $|\xi|\geq 5 \sqrt{N_{E/K}}$, we use the bound $|J_0(x)|\leq \sqrt{2/(\pi|x|)}$ to deduce that
$$ |\widehat Y_2(\xi)|\leq \prod_{j=1}^k \left(\frac{1+q^{-\frac 12}}{(1-q^{-1})^{\frac 12}} \frac{\sqrt{N_{E/K}}}{\pi \xi} \right) \leq \left( \frac{7\sqrt{N_{E/K}}}{\pi \xi} \right)^{k/2},  $$
a quantity which is $\leq \xi^{-4}$ as soon as $|\xi|\geq 5\sqrt{N_{E/K}}$ (we used LI and the fact that $k=(\tfrac 12+o(1))N_{E/K}$, which follows from our assumption on the growth of the rank in terms of $N_{E/K}$).

The last range that we need to treat is $0.35 \sqrt{N_{E/K}} \leq |\xi| \leq 5 \sqrt{N_{E/K}}$; we need to show that $|\widehat Y_2(\xi)| \leq \xi^{-4}$. For these values of $\xi$ we always have that 
$$0.28 \leq \left| \frac{2}{1 - q^{-\frac 12}\ee^{-\ii\theta_j}}\right| \sqrt{\frac{q-1}{q}} \frac{|\xi|}{\sqrt{N_{E/K}}} \leq  35,
 $$
and hence from the properties of the Bessel function $J_0$,
$$ |\widehat Y_2(\xi)|\leq 0.981^{k}. $$
This last quantity is $\leq \xi^{-4}$, since $0.981^k \leq (5N_{E/K})^{-4}$ is clearly true for large enough $N_{E/K}$.

The conclusion of Lemma \ref{lem:berryesseen} is then that the distribution function of $Y_2$ satisfies
\begin{equation}
\sup_{x\in \mathbb R} |F_{Y_2}(x)-G(x)| \ll \frac{\rank(E/K)+\log N_{E/K}}{N_{E/K}}.
\label{eq:conclusionofBE}
\end{equation}  
Now note that $Y_E=Y_1+Y_2$ and that $|Y_1|\ll (\rank(E/K)+1)/\sqrt{N_{E/K}}$ with probability one. Hence, denoting by $Z$ a standard Gaussian random variable, we can apply \eqref{eq:conclusionofBE} to deduce the following about the distribution function of $Y_E$:
\begin{align*} F_E(x) &= \mathbb P[Y_1+Y_2 \leq x] = \mathbb P[Y_1+Y_2 \leq x \,\, \big| \,\,|Y_1| \ll (\rank(E/K)+1)/\sqrt{N_{E/K}}] \\  
& = \mathbb P[Y_2 \leq x+O((\rank(E/K)+1)/\sqrt{N_{E/K}})]  \\
& = \mathbb P[Z \leq x+O((\rank(E/K)+1)/\sqrt{N_{E/K}})]+ O\left(\frac{\rank(E/K)+\log N_{E/K}}{N_{E/K}}\right)\\
&=G(x) + O\left(\frac{\rank(E/K)+1}{\sqrt{N_{E/K}}}\right),
\end{align*}
since $\ee^{-x^2/2}$ is bounded by $1$. Noting that the above calculation is uniform in $x$, the result follows.
\end{proof}

\begin{proof}[Proof of Theorem \ref{FirstTheorem}(ii)]
It is a consequence of Theorem \ref{CentralLimitTheorem}.
\end{proof}

\begin{corollary}\label{cor:LIMeansNoBias}
Assuming LI,
the condition $\rank(E/K) = o(\sqrt{N_{E/K}})$ implies the estimate
$$ \delta(E) = \frac 12 + O\left( \frac{\rank(E/K)+1}{\sqrt{N_{E/K}}} \right)  $$
\end{corollary}

\begin{proof}
Define the random variable $Y_E:=\sqrt{\frac{q-1}{q}}
X_E/\sqrt{N_{E/K}}$, and denote by $F_E(x)$ its distribution function. We have by Theorem \ref{CentralLimitTheorem}
that
$$ 1-\delta(E)= F_E(0) = \frac 12 +O\left( \frac{\rank(E/K)+1}{\sqrt{N_{E/K}}} \right)  , $$
since the Gaussian is symmetric around $0$.

\end{proof}

\begin{proof}[Proof of Theorem \ref{thm:generic bias}]
 
  The proof is obtained by combining Corollary~\ref{cor:LIMeansNoBias} 
  and~\cite[Th.~$2.3$]{CFJ} (choosing $k=1$). Let us briefly recall what the latter result asserts. In the notation set before stating Theorem~\ref{thm:generic bias} 
  let $f\in\mathcal F_d(\Fp_{q^n})$ and
 let $\gamma_{j}(f)$ (seen as complex numbers after fixing an embedding 
 of $\overline{\Qq_\ell}$ in $\Cc$) be the set of inverse roots of $L(E_f/K,T)$ 
 that are \emph{not} forced zeros (see~\eqref{ForcedZeros}) of the $\Qq$-polynomial $L(E_f/K;T)$. By Deligne's purity result the inverse roots $\gamma_j(f)$ all have the same modulus (equal to the cardinality of the subfield of constants in $K$) so we may divide the $\gamma_j(f)$'s by the common modulus and write $\ee^{\ii\theta_j(f)}$, $\theta_j(f)\in[0,2\pi)$, for the complex numbers of modulus $1$ thus obtained. The main object of study in~\cite{CFJ} is the multiplicative $\Zz$-module 
 $$
 {\rm Rel}\left((\gamma_j(f))\right)=\left\{(n_j)\subseteq \Zz\colon \prod_j \ee^{\ii n_j\theta_j(f)}=1\right\}\,,
 $$
which is the set of multiplicative relations among the inverse roots $\gamma_j(f)$. Of course there are multiplicative relations among the $\gamma_j(f)'s$: the ones coming
 from the functional equation satisfied by $L(E_f/K,T)$. In other words 
 these relations come from the invariance of the set of roots of $L(E_f/K,T)$ under inversion.
 In terms of the angles $\theta_j(f)\in [0,2\pi)$, this can be rephrased by saying that for each $j$, $2\pi-\theta_j(f)$ is again the angle of some inverse root of $L(E_f/K,T)$. In~\cite{CFJ} these relations are called \emph{trivial} and the $\Zz$-module
 ${\rm Rel}\left((\gamma_j(f))\right)$ is called trivial if it only consists of trivial relations. From ~\cite[Th.~$2.3$]{CFJ} we know that 
 for all $p$ bigger than a constant depending only on $d$, for all 
  big enough $p$-power 
  $q:=p^m$ (precisely $m$ has to be bigger than a constant depending only on 
  $\overline{\mathcal{F}_d}:=\mathcal F_d\times\overline{\Fp_p}$) and for all $d$ bigger than an absolute constant,
\begin{equation}\label{eq:relations}
q^{-n(d+1)}\#\left\{f\in \mathcal F_d(\Fp_{q^n})\colon \,{\rm Rel}\left(\left(\gamma_{,j}(f)\right)
\right)
\text{ is nontrivial }\right\}\ll nq^{-n\gamma^{-1}}\log q\,,
\end{equation}
where one can take $2\gamma=4+7N_{f}(N_{f}-1)$,
the implied constant depends only on $d$ and the base curve $E/K$, and $N_f$ is the degree of the 
the $\Qq$-polynomial $L(E_f/K,T)$.

Now we claim that
\begin{equation}\label{eq:LowerUpperNf}
d\leq  N_f\leq 2d+C_{E/K}\,,
\end{equation}
where $C_{E/K}$ is a constant depending only on $E/K$.
To see why this holds, first recall (see e.g.~\cite[Th.~1.1(3)]{CFJ} and the references therein) that
$$
N_f=\deg M_f+2\deg A_f+4(g-1)\,.
$$
Recall also that if $\Delta\in K$ is the discriminant of the  minimal Weierstrass model of $E/K$ then the discriminant of the minimal Weierstrass model of $E_f/K$ is $f^6\Delta$ (see~\cite[Section $2.1$]{CFJ}). Thus the locus ${\rm Sing}\, E_f$ of bad reduction of $E_f/K$ consists, besides the locus ${\rm Sing}\, E$ of bad reduction of the base curve $E/K$, of the irreducible factors of $f$. From the above formula for $N_f$, one deduces
$$
\deg\left({\rm Sing}\, E_f\right) \leq N_f\leq 2\deg\left({\rm Sing}\, E_f\right)+4g\,.
$$
We have
$$
\sum_{\pi\mid f}[\Fp_{q^{\deg\pi}}:\Fp_q]\leq\deg\left({\rm Sing}\,E_f\right)\leq \deg\left({\rm Sing}\,E\right)+\sum_{\pi\mid f}[\Fp_{q^{\deg\pi}}:\Fp_q]
$$
where the summation is over irreducible factors $\pi$ of $f$. Obviously this sum  equals $\deg f=d$. The claim follows. Inserting the upper bound in~\eqref{eq:relations}  we obtain
\begin{equation}\label{eq:relations2}
q^{-n(d+1)}\#\left\{f\in \mathcal F_d(\Fp_{q^n})\colon \,{\rm Rel}\left(\left(\gamma_{,j}(f)\right)
\right)
\text{ is nontrivial }\right\}\ll nq^{-nd^{-2}(14+b_E/d)^{-1}}\log q\,,
\end{equation}
under the same conditions as in~\eqref{eq:relations}, and where $b _E$ is a constant depending only on $E$.

Next we use Corollary~\ref{cor:LIMeansNoBias}. To do so we note that if 
the $\Zz$-module ${\rm Rel}(\left(\gamma_{,j}(f)\right)_j)$ is trivial then $E_f/K$ 
satisfies LI
and $\rank(E_f)\leq 1$. Indeed any non-trivial $\Qq$-linear relation\footnote{Note that the tuple of coefficients $(r_1,0,\ldots,0)$, $r_1\in\Qq^\times$, does not give a $\Qq$-linear relation.} involving elements of  $\{1\}\cup\{\theta_j(f)/\pi\in [0,1]\}$ immediately leads (by clearing denominators and exponentiating) to a non-trivial relation in ${\rm Rel}(\left(\gamma_{,j}(f)\right)_j)$. Moreover if $\rank E_f\geq 2$ then one of the $\theta_j(f)$'s is zero (since the forced zeros contribute at most $1$ to this count) which of course produces a non-trivial relation. This discussion shows that~\eqref{eq:relations2} implies
\begin{equation}
\#\left\{f\in \mathcal F_d(\Fp_{q^n})\colon \,E_f/K\text{ violates  LI
or} \rank(E_f)\geq 2\right\}\ll \frac{nq^{n(d+1)}\log q}{q^{nd^{-2}(14+b_E/d)^{-1}}}\,,
\end{equation}
under the same conditions as in~\eqref{eq:relations2}. Theorem~\ref{thm:generic bias} follows from combining this with Corollary~\ref{cor:LIMeansNoBias} and the lower bound in~\eqref{eq:LowerUpperNf}.
 \end{proof}

\section*{Acknowledgments}
We would like to thank Andrew Granville for introducing the second author to elliptic curve prime number races and for suggesting to work on the summatory function of the trace of Frobenius. We would also like to thank Guy Henniart for enlightening explanations regarding Artin conductors and Zeev Rudnick for his feedback and comments. The first author is grateful to Lehigh University for their hospitality during his visit where part of this work was done.
The second author was supported by an NSERC Postdoctoral Fellowship, NSF grant DMS-0635607 as well as a Postdoctoral Fellowship from the Fondation Sciences Math\'ematiques de Paris. This work was accomplished while the second author was at the Institute for Advanced Study, the University of Michigan, and Universit\'e Paris Diderot.



%
%
%
%



\end{document}